\documentclass[12pt]{article}

\usepackage{amssymb}
\usepackage{amsthm}
\usepackage{amsmath}

\usepackage{graphicx}
\usepackage[utf8]{inputenc}
\usepackage[table]{xcolor}
\usepackage{subcaption}
\usepackage{adjustbox} 

\usepackage{float}
\usepackage{braket}

\usepackage{caption}

\usepackage{framed}
\usepackage{tabularx}
\usepackage{colortbl}
\usepackage{longtable}
\usepackage{multirow}

\def \cF {\mathcal{F}}
\def \cH {\mathcal{H}}
\def \zzv {\mathbb{Z}_v}

\def \bsk {\bigskip}

\def \kn {K_n^{(3)}}
\def \kv {K_v^{(3)}}
\newcommand{\kpq}[2]{K_{#1\to #2}^{(3)}}
\newcommand{\crt}[2]{K_{#1\leftrightarrow #2}^{(3)}}
\newcommand{\ktri}[1]{K_{#1}^{(3)}}
\newcommand{\kvv}[1]{K_{#1}^{(3)}-I}
\def \kvi {K_v^{(3)}-I}
\def \kvr {K_v^{(r)}}
\def \kvri {K_v^{(r)}-I}
\def \chk {\cH_2}

\newcommand{\CCC}[2]{{\mathcal{C^*}(3,#1,#2)}}

\newtheorem{Theorem}{Theorem}
\newtheorem{Remark}[Theorem]{Remark}
\newtheorem{Lemma}[Theorem]{Lemma}

\newtheorem{Corollary}[Theorem]{Corollary}

\newenvironment{Proof}{\noindent {\it Proof.}}{}

\newenvironment{Proofof}[1]{\noindent {\it Proof of #1.}}{}

 \newcounter{conjecture}
 \newtheorem{Conjecture}[conjecture]{Conjecture}

\definecolor{color5}{HTML}{BCBCBC}
\newcolumntype{f}{>{\columncolor{color5}}c}

\begin{document}

\title{Spectrum of 3-uniform $6$- and $9$-cycle systems
 over $K_v^{(3)} - I$}

\author{Anita Keszler$^1$ and \vspace*{2ex} Zsolt Tuza$^{2,3}$\\
\normalsize $^1$ Machine Perception Research Laboratory \\
\normalsize Institute for Computer Science and Control (SZTAKI) \\
\normalsize 1111 Budapest, Kende u.\ 13--17, \vspace*{2ex} Hungary \\
\normalsize $^2$ Alfr\'ed R\'enyi Institute of Mathematics \\
\normalsize 1053 Budapest, Re\'altanoda u.\ 13--15, \vspace*{2ex} Hungary \\
\normalsize $^3$ Department of Computer Science and Systems Technology \\
\normalsize University of Pannonia \\
\normalsize 8200 Veszpr\'em, Egyetem u.\ 10, Hungary
}
\date{}
\maketitle



\begin{abstract}
We consider edge decompositions of $K_v^{(3)} - I$, the complete
 3-uniform hypergraph of order $v$ minus a 1-factor
 (parallel class, packing of $v/3$ disjoint edges).
We prove that a decomposition into tight 6-cycles exists if and only if
 $v\equiv 0,3,6$ (mod 12) and $v\geq 6$; and a decomposition into
 tight 9-cycles exists for all $v\geq 9$ divisible by 3.
These results are complementary to the theorems of Akin et al.\
 [Discrete Math.\ 345 (2022)] and Bunge et al.\
  [Australas.\ J. Combin.\ 80 (2021)].

\bsk


\bsk


\end{abstract}


\section{Introduction}
\label{sec:intro}

A \emph{decomposition} of a graph or hypergraph $\cH$ is a collection of sub\-(hyper)\-graphs $\cH_1,\dots,\cH_m$ such that $\cH$ is their edge-disjoint union.
If all $\cH_i$ are isomorphic to a fixed hypergraph $\cF$, then $\{\cH_1,\dots,\cH_m\}$ is called an \emph{$\cF$-decomposition} of $\cH$.
Here we study decompositions of nearly complete 3-uniform hypergraphs into cycles of given lengths 6 and 9.

Edge decompositions of complete graphs $K_v$, and complete graphs minus a 1-factor, $K_v-I$, of order $v$ have a long history for over a century.
Concerning the existence of decompositions into cycles of a fixed length it was proved by Alspach and Gavlas \cite{AG_2001} and \v Sajna \cite{S_2002} that the standard necessary arithmetic conditions---i.e., the degree $v-1$ or $v-2$ must be even, and the number $\binom{v}{2}$ or $\frac{v(v-2)}{2}$ of edges must be a multiple of the cycle length---are also sufficient.
Also the existence of decompositions into Hamiltonian cycles requires just the proper parity of $v$; this well-known fact dates back to the 19th century.

\subsection{Berge cycles in hypergraphs}
\label{subsec:berge}

For complete 3-uniform hypergraphs $\kv$ of order $v$, as a first step of generalization, the situation becomes more complicated.
In hypergraphs there are several ways to define cycles, and the stricter one we take the harder the question of decomposability becomes. 

The weakest requirement is imposed on \emph{Berge $k$-cycles} that are alternating cyclic sequences $v_1,e_1,v_2,e_2,\dots,v_k,e_k$ of $k$ mutually distinct vertices and edges such that $v_i\in e_i\cap e_{i-1}$ for all $1<i\leq k$ and $v_1\in e_1\cap e_k$.
The particular case of $k=n$ means Hamiltonian Berge cycle; the existence of decompositions of $\kv$ into such cycles was proved for $v\equiv 2,4,5$ (mod 6) by Bermond \cite{B_1978} and for the other cases by Verall \cite{V_1994}.
K\"uhn and Osthus \cite{KO_2014} generalized this result for
 complete $r$-uniform hypergraphs $\kvr$ ($r\geq 4$), with the
 slight restriction $v\geq 30$ if $r=4$ and $v\geq 20$ if $r\geq 5$.
Analogous results are valid for $\kvi$ and more generally for
 $\kvri$ as well; these nearly complete hypergraphs are obtained
 from $\kvr$ by the deletion of a 1-factor (also called parallel
 class in some context), i.e.\ omitting $v/r$ mutually disjoint edges.
(This requires the assumption that the number $v$ of vertices is a
 multiple of the edge size $r$.)
Moreover, for cycles of given length the decomposition problem on
 $\kvr$ (and even for complete multi-hypergraphs) was solved by
 Javadi, Khodadadpour and Omidi \cite{JKO_2018} for all $v\geq 108$.
For cycle lengths $k=4$ and $k=6$ with edge size $r=3$ it is also
 known that the lower bounds on $v$ can be omitted; see
  \cite{JN_2018} and \cite{LP_2021}, respectively.

\subsection{Tight cycles in uniform hypergraphs}
\label{subsec:tightcycles}

A stricter cycle definition is \emph{$r$-uniform tight\/ $k$-cycle}
 that means a cyclic sequence $v_1,v_2,\dots,v_k$ of
 $k$ vertices, together with $k$ edges
 formed by the
 $r$-tuples of consecutive vertices $v_i,v_{i+1},\dots,v_{i+r-1}$
  ($i=1,2,\dots,k$, subscript addition taken modulo $k$).

\emph{From now on, by\/ \emph{$k$-cycle} we mean $3$-uniform
 tight $k$-cycle.}
The decomposition problem into such cycles seems much harder than
 the one above on Berge cycles, already on $\kv$ and $\kvi$.
Decomposability of $\kv$ into Hamiltonian cycles has the simple
 necessary condition $3\mid (n-1)(n-2)$.
But its sufficiency has been studied only within a limited range
 ($v\leq 16$ by Bailey and Stevens \cite{BS_2010},
 $v\leq 32$ by Meszka and Rosa \cite{MR_2009},
 $v\leq 46$ by Huo et al.\ \cite{H-etal_2015}.)

For fixed cycle length $k$ concerning $\kv$, only three cases are
 solved: the very famous class of Steiner Quadruple Systems ($k=4$)
 by the classical theorem of Hanani \cite{H_1960}, and the
 very recent works by Akin et al.\ \cite{A-etal_2022} for $k=6$ and
 by Bunge et al.\ \cite{B-etal_2021} for $k=9$.
There are many constructions for $k=5$ and $k=7$, but no complete
 solution is available on them; for partial results and further
 references we cite \cite{KT_2021} and \cite{MMGJ_2019}.

Johnson \cite{J_1962} and Sch\"onheim \cite{S_1966} proved general
 upper bounds from which it follows that a packing of edge-disjoint
 4-cycles (i.e., a partial SQS$(v)$ system) cannot contain more than
$$
  \left\lfloor \frac{v}{4} \left\lfloor \frac{v-1}{3}
   \left\lfloor \frac{v-2}{2} \right\rfloor \right\rfloor
    \right\rfloor
$$
 4-cycles.
For $v\equiv 3$ (mod 6) this means $v(v-1)(v-3)/24$.
Moreover, Brouwer \cite{B_1977} emphasizes that another, essentially
 forgotten upper bound due to Johnson yields $v(v^2-3v-6)/24$ if
 $v\equiv 0$ (mod 6).
Hence, in either case the number of edges covered by any collection
 of 4-cycles is smaller than $v^2(v-3)/6$.
In our context this fact has the following
 important consequence.

\begin{Corollary}
 No\/ $\kvi$ can admit a decomposition
  into tight\/ $4$-cycles.
\end{Corollary}

The goal of the present note is to prove that in the other two 
 cases that are solved for $\kv$, namely $k=6$ and $k=9$, the
 $k$-cycle decompositions of $\kv$ have their natural analogues
 for $\kvi$.
Hence we solve the spectrum problem for the decomposability
 of $\kvi$ for the cases of 6-cycles and 9-cycles.
More explicitly, we prove the following two results.

\begin{Theorem}
\label{t:c6}
The hypergraph\/ $\kvi$ admits a decomposition into\/ $6$-cycles
 if and only if\/ $v\geq 6$ and\/ $v\equiv 0,3,6$ {\rm (mod 12)}.
\end{Theorem}

\begin{Theorem}
\label{t:c9}
The hypergraph\/ $\kvi$ admits a decomposition into\/ $9$-cycles
 if and only if\/ $v\geq 9$ and\/ $v$ is a multiple of\/ $3$.
\end{Theorem}

\begin{Remark}
Since\/ $\kvi$ has\/ $v^2(v-3)/6$ edges, and every\/ $k$-cycle has
 exactly\/ $k$ edges, for the existence of a decomposition into\/
 $k$-cycles we must have\/ $6k\mid v^2(v-3)$, therefore
 the conditions given in Theorems \ref{t:c6} and \ref{t:c9} are
 necessary.
\end{Remark}

Sufficiency of the conditions in Theorems \ref{t:c6} and \ref{t:c9}
 will be proved in Section~\ref{s:c6} and Section \ref{s:c9}, respectively.
Further systems with additional properties are constructed in Sections \ref{sec:2_split_sys} and \ref{sec:concl}.
One of those properties specifies ``2-split systems'' \cite{GMT_2020}
 that have been used frequently in recursive constructions
 for other cycle lengths.
The other considered type is ``cyclic systems'' having a
 rotational symmetry.

\subsection{Notation}

We write $\CCC{k}{v}$ to denote any decomposition of
 $\kvi$ into tight 3-uniform $k$-cycles.
Moreover, for some particular types of 3-uniform
 hypergraphs we use the following notation
  (see Figure \ref{fig:Kab} for illustrations):
\begin{itemize}
  \item $\ktri{a,b,c}$ --- complete 3-partite hypergraph whose
    vertex set is partitioned into three sets $A,B,C$ with
     $|A|=a$, $|B|=b$, $|C|=c$, and a 3-element set
    $T\subset A\cup B\cup C$ is an edge if and only if
    $|T\cap A|=|T\cap B|=|T\cap C|=1$.
  \item $\kpq{a}{b}$ --- 3-uniform hypergraph whose vertex set is
    partitioned into two sets $A$ and $B$ with $|A|=a$ and
    $|B|=b$, and a 3-element set $T\subset A\cup B$ is an
    edge if and only if $|T\cap A|=2$ and $|T\cap B|=1$.
  \item $\crt{a}{b}$ --- a hypergraph of ``crossing triplets'':
    3-uniform hypergraph whose vertex set is
    partitioned into two sets $A$ and $B$ with $|A|=a$ and
    $|B|=b$, and a 3-element set $T\subset A\cup B$ is an
    edge if and only if it meets both $A$ and $B$.
\end{itemize}

\begin{figure}
\centering
\begin{subfigure}{0.31\textwidth}
    \fbox{\includegraphics[width=\textwidth]{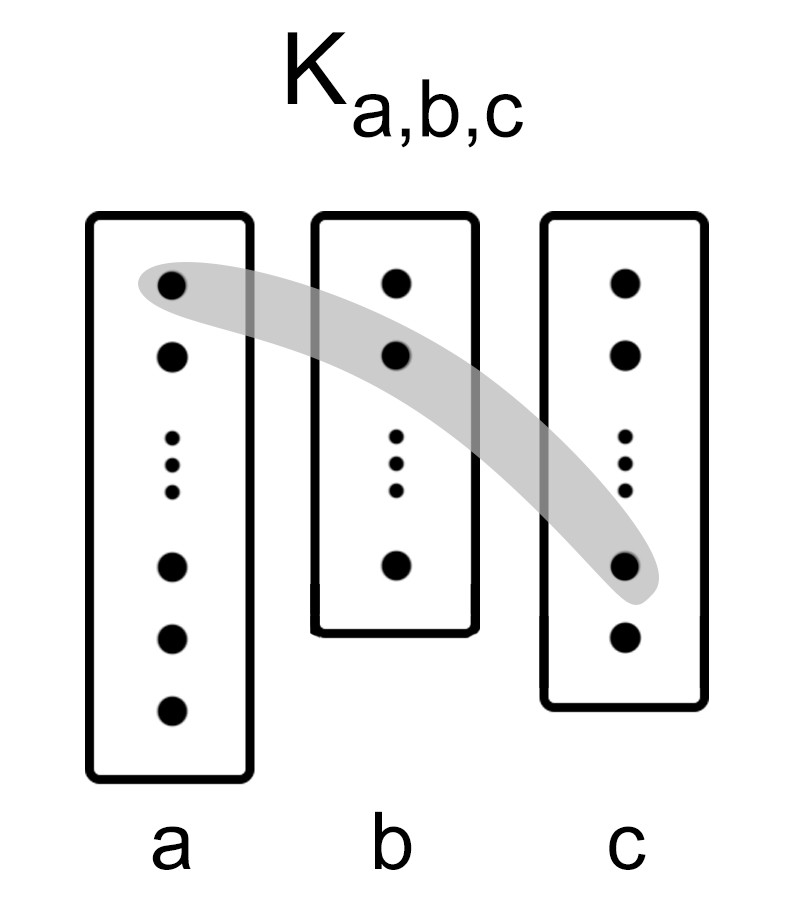}}
    \caption{$\ktri{a,b,c}$}
    \label{fig:Kab_1}
\end{subfigure}
\hfill
\begin{subfigure}{0.31\textwidth}
    \fbox{\includegraphics[width=\textwidth]{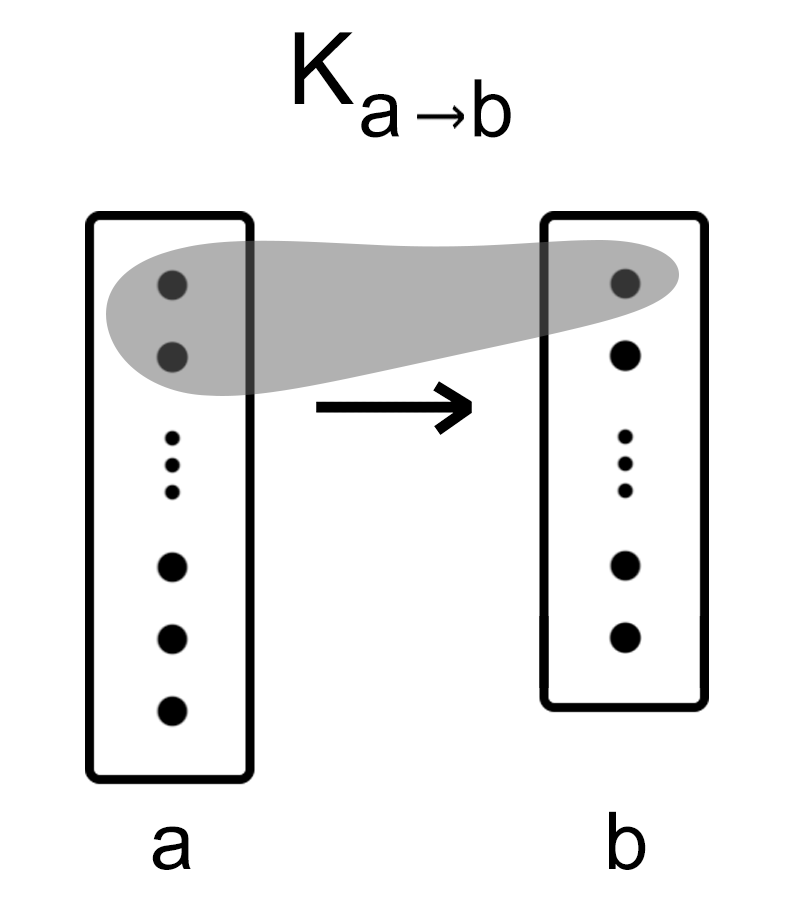}}
    \caption{$\kpq{a}{b}$}
    \label{fig:Kab_2}
\end{subfigure}
\hfill
\begin{subfigure}{0.31\textwidth}
    \fbox{\includegraphics[width=\textwidth]{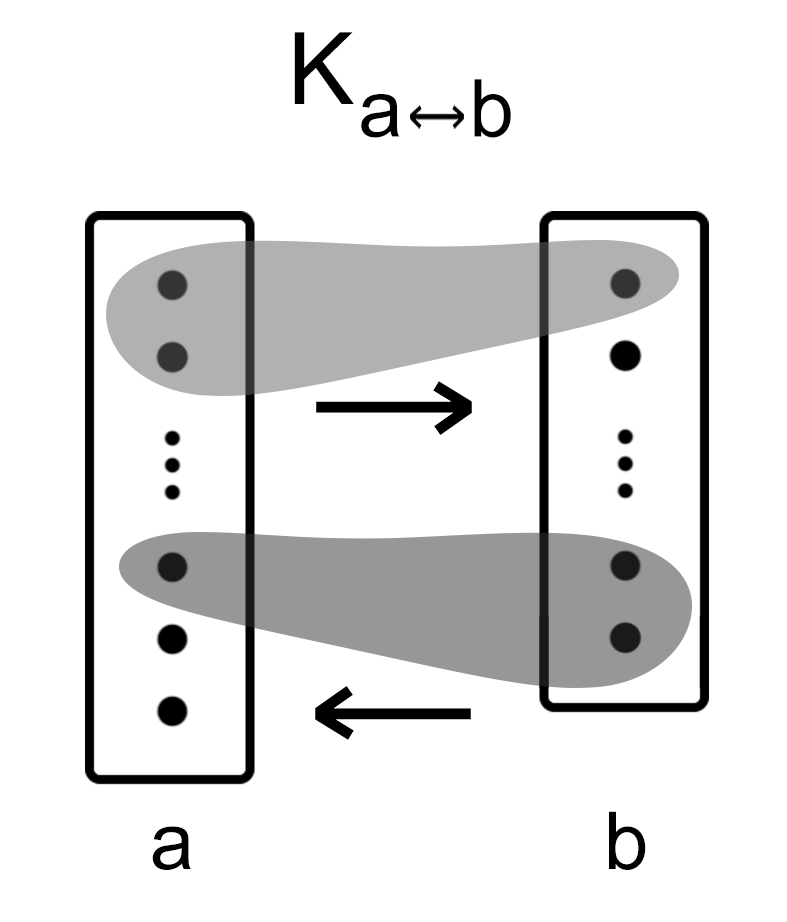}}
    \caption{$\crt{a}{b}$}
    \label{fig:Kab_3}
\end{subfigure}
\hfill
\caption{Illustrations for $\ktri{a,b,c}$, $\kpq{a}{b}$ and $\crt{a}{b}$.}
\label{fig:Kab}
\end{figure}


\section{The spectrum of $\CCC{6}{v}$ systems}
\label{s:c6}

In this section we prove Theorem \ref{t:c6}.
Let $v=12m+3s$, where $m\geq 1$ is any integer and $s=0,1,2$.
We denote by $\chk$ the 3-uniform hypergraph with four vertices
 $a,b,c,d$ and two edges $abc,abd$.
We shall apply the following well-known result.

\begin{Theorem}[Bermond, Germa, Sotteau \cite{BGS_1977}]
\label{t:H2} 
If\/ $n\equiv 0,1,2$ {\rm (mod 4)}, then\/ $\kn$ has a\/
 $\chk$-decompo\-si\-tion.
\end{Theorem}

Moreover, two building blocks will be used.
The first one is derived from the cycle double cover of the edge set
 of the complete bipartite graph $K_{3,3}$ with the three cycles
 $x_1y_1x_2y_2x_3y_3$, $y_1x_2y_4x_1y_2x_3$, $x_2y_3x_3y_1x_1y_2$,
  where the two vertex classes are $\{x_1,x_2,x_3\}$ and
   $\{y_1,y_2,y_3\}$.

\begin{Lemma}
 \label{l:bi6}
The hypergraph\/ $\kvv{6}$ obtained from\/ $\ktri{6}$ by omitting
 the\/ $1$-factor\/ $I = \{ a_1a_2a_3, a_4a_5a_6 \}$ has a
  decomposition into the three\/ $6$-cycles
  $$
    a_1 a_4 a_2 a_5 a_3 a_6 \, , \quad
    a_4 a_2 a_6 a_1 a_5 a_3 \, , \quad
    a_2 a_6 a_3 a_4 a_1 a_5 \, .
  $$
\end{Lemma}

The second small construction is derived by combining the
 two cyclic $P_4$-decompositions
 $x_{1+i}y_{1+i}x_{2+i}y_{3+i}$ and
  $y_{1+i}x_{1+i}y_{3+i}x_{2+i}$ ($i=0,1,2$, subscript addition
 taken modulo 3) of the same $K_{3,3}$.

\begin{Lemma}
\label{l:tri8}
The complete\/ $3$-partite hypergraph\/ $\ktri{3,3,2}$ with its three
 vertex classes\/ $\{a_1,a_2,a_3\}$, $\{b_1,b_2,b_3\}$, $\{c_1,c_2\}$
 has a decomposition into the three\/ $6$-cycles
  $$
    a_{1+i} \, b_{1+i} \, c_1 \, a_{2+i} \, b_{3+i} \, c_2  \quad
    (i = 0, 1, 2)
  $$
 where subscript addition is taken modulo\/ $3$.
\end{Lemma}

Alternative constructions for Lemmas \ref{l:bi6}
 and \ref{l:tri8}, which have a less symmetric structure,
 can be found in Examples 2 and 5 of \cite{A-etal_2022}.

\bsk

\begin{Proofof}{Theorem \ref{t:c6}}

The case of $v=6$ is settled in Lemma \ref{l:bi6}.
Let now $v=12u+3s$, where $u\ge 1$ and $s=0,1,2$.
We view $\kvi$ as the result of substituting $4u+s$ triplets
 $A_1,\dots,A_{4u+s}$ into the vertices $x_1,\dots,x_{4u+s}$ of
 $\ktri{4u+s}$ (see Fig.~\ref{fig:4u_s}).
The sets $A_i$ will play the role of edges in the 1-factor whose
 omission from the edge set of $\kv$ yields $\kvi$.
Then each edge of $\kvi$ meets two or three of the sets $A_i$.
Those two types of edges will be covered with 6-cycles separately.

\begin{figure}
\centering
\begin{subfigure}{0.55\textwidth}
    \fbox{\includegraphics[width=\textwidth]{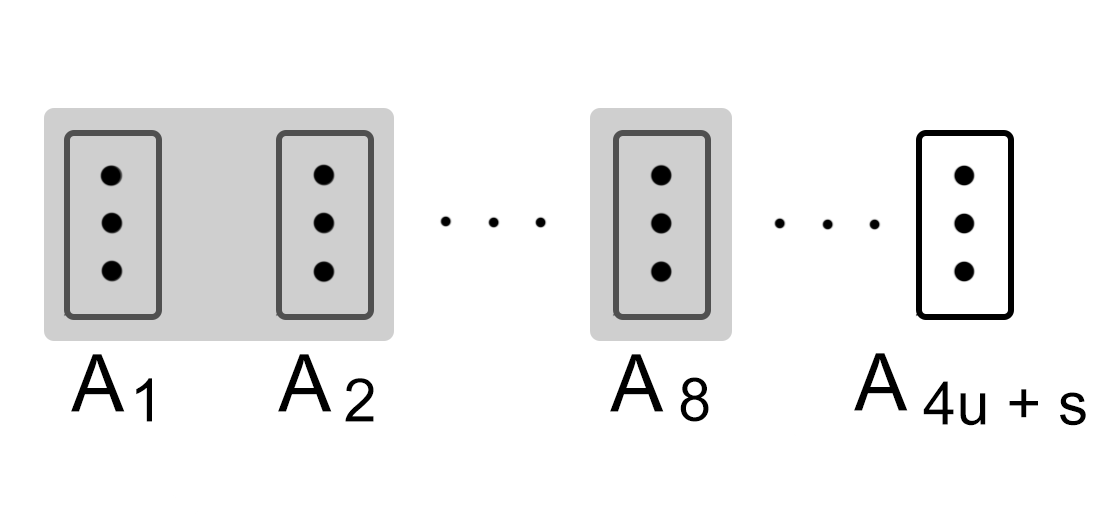}}
    \caption{The $4u+s$ triplets $A_1,\dots,A_{4u+s}$, with the example triplets $A_1, A_2, A_8$ highlighted.}
    \label{fig:4u_s_1}
\end{subfigure}
\hfill
\begin{subfigure}{0.30\textwidth}
    \fbox{\includegraphics[width=\textwidth]{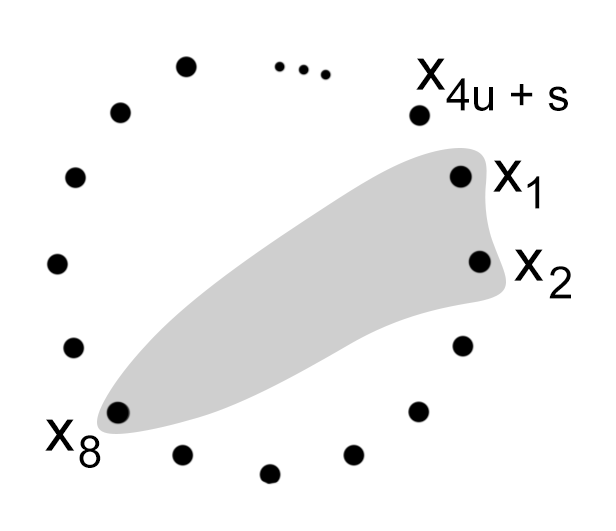}}
    \caption{The vertices $x_1,\dots,x_{4u+s}$ of $\ktri{4u+s}$. $A_1, A_2, A_8 \rightarrow x_1,x_2,x_8$}
    \label{fig:4u_s_2}
\end{subfigure}
\hfill
\caption{We view $\kvi$ as the result of substituting $4u+s$ triplets
 $A_1,\dots,A_{4u+s}$ into the vertices $x_1,\dots,x_{4u+s}$ of
 $\ktri{4u+s}$. An example with triplets $A_1, A_2, A_8$ and the corresponding three vertices  $x_1,x_2,x_8$ is highlighted.}
\label{fig:4u_s}
\end{figure}

\medskip

(a)\quad
For the edges meeting two of the $A_i$ we consider all pairs
 $i',i''$ with $1\leq i'<i''\leq 4u+s$, and apply Lemma \ref{l:bi6}
  to cover all triplets but $A_{i'}$ and $A_{i''}$ inside
  $A_{i'}\cup A_{i''}$.
This step creates three 6-cycles for each pair $i',i''$.

\medskip

(b)\quad
For the edges meeting three of the $A_i$ we take an $\chk$-decomposition
 of $\ktri{4u+s}$ as guaranteed by Theorem \ref{t:H2}.
Suppose that the triplets $x_ix_jx_{k'}$, $x_ix_jx_{k''}$ form a
 copy of $\chk$ in this decomposition (see Fig.~\ref{fig:H2_1}).
They represent the complete 3-partite hypergraph $\ktri{3,3,6}$
 whose vertex classes are $A_i$, $A_j$, and $A_{k'}\cup A_{k''}$.
We split $A_{k'}\cup A_{k''}$ into three pairs
 $C_{k_1},C_{k_2},C_{k_3}$, in this way decomposing $\ktri{3,3,6}$
 into three copies of $\ktri{3,3,2}$ (see Fig.~\ref{fig:H2_2}).
Now Lemma \ref{l:tri8} can be applied to find a
 6-cycle decomposition of the complete 3-partite hypergraphs whose
 vertex classes are $A_i$, $A_j$, and $C_{k_\ell}$ for $\ell=1,2,3$.
This step creates nine 6-cycles for each copy of $\chk$.

\begin{figure}
\centering
\begin{subfigure}{0.47\textwidth}
    \fbox{\includegraphics[width=\textwidth]{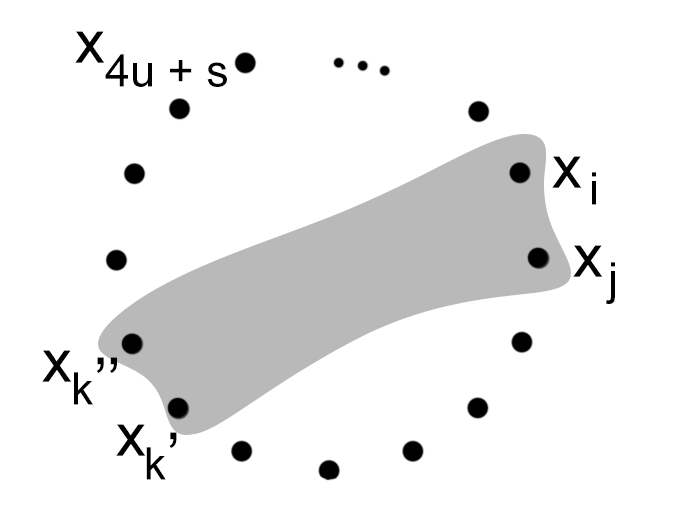}}
    \caption{The vertices  $x_i,x_j,x_{k'},x_{k''}$ that belong to a copy of $\chk$.}
    \label{fig:H2_1a}
\end{subfigure}
\hfill
\begin{subfigure}{0.47\textwidth}
    \fbox{\includegraphics[width=\textwidth]{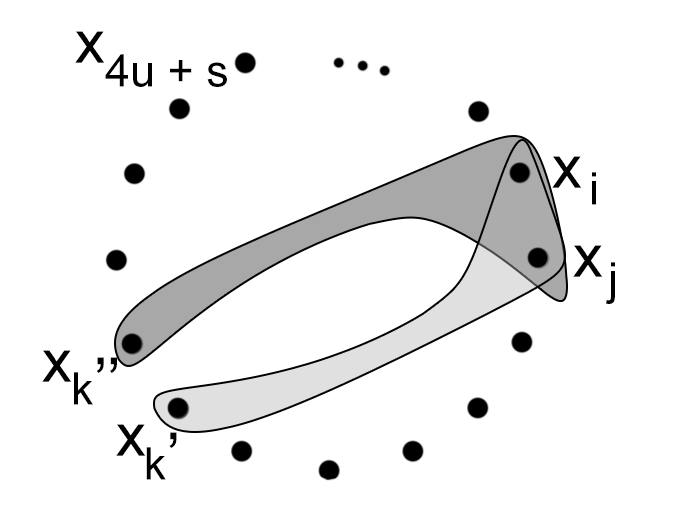}}
    \caption{The triplets (hyperedges)  $x_ix_jx_{k'}$, $x_ix_jx_{k''}$ forming a copy of  $\chk$.}
    \label{fig:H2_1b}
\end{subfigure}
\caption{Example $\chk$ formed by the triplets  $x_ix_jx_{k'}$, $x_ix_jx_{k''}$ in the decomposition.}
\label{fig:H2_1}
\end{figure}

\begin{figure}[t]
\centering
\begin{subfigure}{0.47\textwidth}
    \fbox{\includegraphics[width=\textwidth]{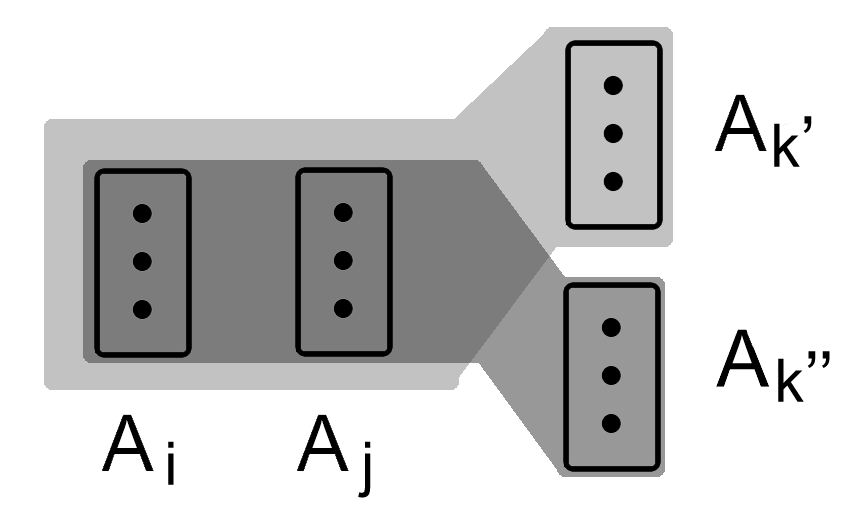}}
    \caption{$A_i$, $A_j$, $A_{k'}$ and $A_{k''}$ are the corresponding sets of the $\chk$, forming a $\ktri{3,3,6}$.}
    \label{fig:H2_2a}
\end{subfigure}
\hfill
\begin{subfigure}{0.47\textwidth}
    \fbox{\includegraphics[width=\textwidth]{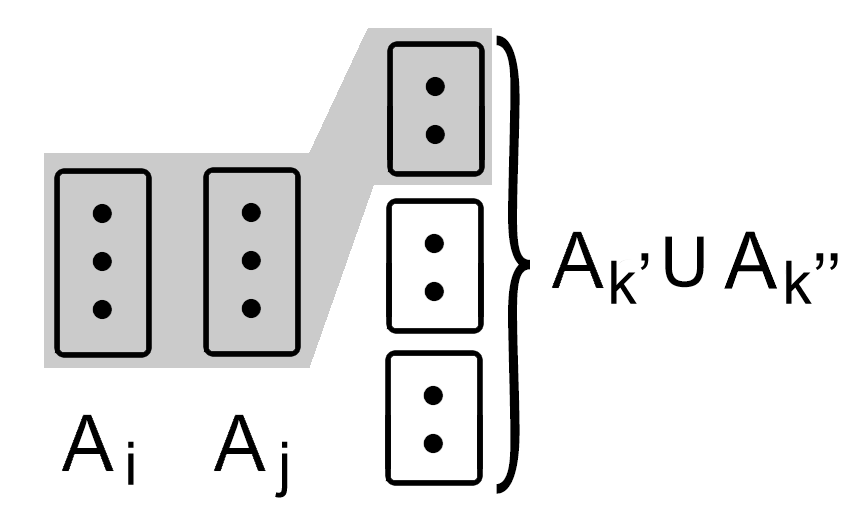}}
    \caption{The $\ktri{3,3,6}$ can be decomposed into three copies of $\ktri{3,3,2}$.}
    \label{fig:H2_2b}
\end{subfigure}
\hfill
\caption{A copy of $\chk$ represents the complete 3-partite hypergraph $\ktri{3,3,6}$. Continuing the example of Fig.~\ref{fig:H2_1}, the  vertex classes of the  $\ktri{3,3,6}$ are $A_i$, $A_j$ and  $A_{k'}\cup A_{k''}$. We split $A_{k'}\cup A_{k''}$ into three pairs, decomposing $\ktri{3,3,6}$ into three copies of $\ktri{3,3,2}$.}
\label{fig:H2_2}
\end{figure}

\medskip

These collections of 6-cycles decompose $\kvi$.
\qed

\end{Proofof}


\section{The spectrum of $\CCC{9}{v}$ systems}
\label{s:c9}

In this section we prove Theorem \ref{t:c9}.
Although its form is simpler than Theorem \ref{t:c6}, it requires more types of building blocks than the construction for 6-cycles, moreover a distinction between the three residue classes will also be needed.
Before the main part of the proof we give several constructions of 9-cycle decompositions of hypergraphs on a small number of vertices.

\begin{Lemma}
\label{l:k9c9}
There exists a decomposition of\/ $\kvv{9}$ into nine\/ $9$-cycles.
\end{Lemma}

\begin{Proof}
We construct a system with cyclic symmetry, composed from the following 9-cycles, subscript addition taken modulo 9:
 $$
   a_{1+i} \, a_{2+i} \, a_{3+i} \, a_{8+i} \, a_{7+i} \,
    a_{5+i} \, a_{9+i} \, a_{6+i} \, a_{4+i} \quad
    (i=0,1,\dots,8) \, .
 $$
These cycles cover all vertex triplets but $a_1a_4a_7$, $a_2a_5a_8$, and $a_3a_6a_9$.
\qed
\end{Proof}

\bsk

The method of the following construction works in a more general way also, for an infinite sequence of cycle lengths as shown in \cite{KT_2022+}; but here we only need its particular case yielding 9-cycles.

\begin{Lemma}
\label{l:9+3c9}
There exists a decomposition of\/ $\kpq{9}{3}$ into twelve\/ $9$-cycles.
\end{Lemma}

\begin{Proof}
Let $A=\{a_1,\dots,a_9\}$ and $B=\{b_1,b_2,b_3\}$ be the two vertex
 classes of $\kpq{9}{3}$.
As mentioned in the introduction, cycle systems on ordinary complete
 graphs $K_v$ exist whenever $v$ is odd and the number of edges is
 divisible by the given cycle length.
In particular, $K_9$ with 9 vertices and 36 edges can be decomposed
 into six 6-cycle subgraphs.
We take this auxiliary decomposition over the vertex set $A$, and
 construct two 9-cycles in $\kpq{9}{3}$ for each of its graph 6-cycles.
Say, $a_1a_2a_3a_4a_5a_6 $ is one of the cycles in $K_9$.
We define the two 9-cycles
 $$
   a_1 \, b_1 \, a_2 \, a_3 \, b_2 \, a_4 \, a_5 \, b_3 \, a_6 \, ,\quad
   b_2 \, a_1 \, a_2 \, b_3 \, a_3 \, a_4 \, b_1 \, a_5 \, a_6 \, .
 $$
The 18 edges of size 3 in these two 9-cycles are precisely those
 triplets that consist of two consecutive vertices along the graph
 6-cycle and one vertex from $B$.
Hence, taking the same for all the six cycles in the decomposition
 of $K_9$, a required collection of 9-cycles is obtained in $\kpq{9}{3}$.
\qed
\end{Proof}

\bsk

A similarly symmetric construction cannot be expected for $\kpq{6}{3}$,
 nevertheless a 9-cycle decomposition still exists.

\begin{Lemma}
\label{l:6+3c9}
The five\/ $9$-cycles\hfill\break
\centerline{$a_1	 \, a_2	\, b_1 \, a_3 \, a_4 \, b_2 \, a_5 \, a_6 \, b_3$}
\centerline{$a_1 \, a_5	\, b_1 \, a_6 \, a_4 \, b_2 \, a_2 \, a_3 \, b_3$}
\centerline{$a_2 \, a_4 \, b_1 \, a_1 \, a_6 \, b_2 \, a_3 \, a_5 \, b_3$}
\centerline{$a_3 \, a_6 \, b_1 \, a_2 \, a_5 \, b_2 \, a_1 \, a_4 \, b_3$}
\centerline{$a_4 \, a_5 \, b_1 \, a_3 \, a_1 \, b_2 \, a_2 \, a_6 \, b_3$}
 decompose\/ $\kpq{6}{3}$ whose two vertex classes are\/
 $A=\{a_1,a_2,\dots,a_6\}$ and\/ $B=\{b_1,b_2,b_3\}$.
\end{Lemma}

We also recall the following construction, which is derived from three Hamiltonian cycles of the complete graph $K_9$.

\begin{Lemma}[Bunge et al.~\cite{B-etal_2021}, Example 5]
\label{l:333c9} 
The three\/ $9$-cycles
 $$
   a_1 \, a_{1+i} \, a_{1+2i} \dots a_{1+8i} \quad (i=1,2,4)
 $$
 with subscript addition modulo\/ $9$ decompose\/ $\ktri{3,3,3}$
 whose vertex classes are\/ $\{a_1,a_4,a_7\}$,\/ $\{a_2,a_5,a_8\}$,
  and\/ $\{a_3,a_6,a_9\}$.
\end{Lemma}

It will be convenient to put these small structures together in some larger hypergraphs as follows.

\begin{Lemma}
\label{l:9+9c9}
All of the following types of hypergraphs admit decompositions
 into\/ $9$-cycles:
 \begin{itemize}
  \item[$(i)$] $\ktri{3p,3q,3r}$ for all\/ $p,q,r\ge 1$,
  \item[$(ii)$] $\kpq{3p}{3q}$ for all\/ $p\ge 2$ and\/ $q\ge 1$,
  \item[$(iii)$] $\crt{3p}{3q}$ for all\/ $p,q\ge 2$,
 \end{itemize}
where\/ $p,q,r$ denote integers.
\end{Lemma}

\begin{Proof}
Concerning $(i)$, the following decomposition chain is easily seen:
 $$
   \ktri{3p,3q,3r} \longrightarrow r \times \ktri{3p,3q,3}
    \longrightarrow qr \times \ktri{3p,3,3}
    \longrightarrow pqr \times \ktri{3,3,3}
 $$
  and for $\ktri{3,3,3}$ we have a decomposition into three 9-cycles
   by Lemma \ref{l:333c9}.
Similarly in $(ii)$ we can do the step
 $\kpq{3p}{3q} \longrightarrow q \times \kpq{3p}{3}$.
Now we write $p$ in the form $p=2a+3b$, which is possible whenever
 $p\geq 2$.
(In fact $b=0$ or $b=1$ can always be ensured.)
Splitting the $3p$ vertices in the first class of $\kpq{3p}{3}$ into
 $a$ sets of size 6 and $b$ sets of size 9 we can proceed with the step
 $$
   \kpq{3p}{3} \longrightarrow 
	\left( a \times \kpq{6}{3} \right)
	\cup \left( b \times \kpq{9}{3} \right)
	\cup \left( ab \times \ktri{3,6,9} \right)
	\cup \left( \binom{a}{2}   \times \ktri{3,6,6} \right)
	\cup \left( \binom{b}{2}  \times \ktri{3,9,9} \right)
 $$
  completing the decomposition by Lemmas
  \ref{l:9+3c9} and \ref{l:6+3c9},
 also using $(i)$.
Finally, $(iii)$ is implied by
 $\crt{3p}{3q}  \longrightarrow \left( \kpq{3p}{3q} \right)
  \cup \left( \kpq{3q}{3p} \right)$ that can be done due to $(ii)$.
\qed
\end{Proof}

\bigskip

We shall need two further small cases for the general proof
 of Theorem~\ref{t:c9}.
The first one is $v=12$.

\begin{Lemma}
\label{l:12c9}
For\/ $v=12$ the two\/ $9$-cycles
 $$
   a_1 \, a_2 \, a_3 \, a_5 \, a_6 \, a_9 \, a_{11} \, a_4 \, a_7 \, ,\quad
   a_1 \, a_2 \, a_8 \, a_{12} \, a_3 \, a_{10} \, a_{11} \, a_7 \, a_9
 $$
 and their rotations modulo\/ $12$ form a decomposition of\/
 $\kvv{12}$, with uncovered (omitted) triplets\/ $a_1a_5a_9$,
 $a_2a_6a_{10}$, $a_3a_7a_{11}$, $a_4a_8a_{12}$.
\end{Lemma}

As a final auxiliary step, we also need to handle the case of $v=15$ separately.

\begin{Lemma}
\label{l:15c9}
The hypergraph\/ $\kvv{15}$ admits a\/ $9$-cycle decomposition
 into\/ $50$ cycles.
\end{Lemma}

\begin{Proof}
Let the vertex set be $A\cup B\cup C$, with $|A|=|B|=6$ and $|C|=3$.
Inside $A\cup C$ and also inside $B\cup C$ we take a copy of the
 9-cycle decomposition of $\kvv{9}$, where $C$ is an uncovered
 vertex triplet.
These 18 cycles cover all triplets
 inside $A$ and inside $B$---with the exception of two disjoint
 triplets in each---and
   also the $A$--$C$ and $B$--$C$ crossing triplets.
The $A$--$B$ crossing triplets can be covered by a decomposition of
 $\crt{6}{6}$ (20 cycles); and the triplets meeting all the three
  parts $A,B,C$ decompose as $\ktri{6,6,3}$ (12 cycles), applying Lemma \ref{l:9+9c9} for both cases.
\qed
\end{Proof}

\bigskip

Now we are in a position to prove Theorem \ref{t:c9}.

\bigskip

\noindent
\begin{Proofof}{Theorem \ref{t:c9}}

The three feasible residue classes $0,3,6$ modulo 9 will be treated separately.

\medskip

\noindent
{\sl Case 1:\/ $v\equiv 0$ (mod 9).}

For $v=9u$ let the vertex set be $A_1\cup \cdots \cup A_u$ with
 $|A_i|=9$ for all $1\leq i\leq u$.
Inside each $A_i$ we take a copy of the system given in
 Lemma \ref{l:k9c9}, a 9-cycle decomposition of $\kvv{9}$.
The family of triplets meeting an $A_i$ in two vertices and another
 $A_j$ in one vertex can be covered by the 9-cycle decomposition of
 $\crt{9}{9}$, as in Lemma \ref{l:9+9c9}.
Finally, the triplets meeting three distinct parts $A_i,A_j,A_k$
 form a copy of $\ktri{9,9,9}$ for any $1\leq i<j<k\leq u$, hence
 this type is also decomposable into 9-cycles by Lemma \ref{l:9+9c9}.

\medskip

\noindent
{\sl Case 2:\/ $v\equiv 3$ (mod 9).}

If $v=9u+3$, beside the sets $A_1,\dots,A_u$ with $|A_i|=9$ we
 also take an $A_0$ with $|A_0|=3$.
Now inside each $A_0\cup A_i$ we insert a copy of the $\kvv{12}$
 decomposition as in Lemma \ref{l:12c9}, in such a way that $A_0$
 is one of the uncovered triplets.
Hence the copies of $\kvv{12}$ for distinct $i$ are independent of
 each other.
Inside $A_1\cup \cdots \cup A_u$ we cover the triplets meeting
 more than one $A_i$ in the same way as we did in the case of
 $v=9u$.
Hence only those triplets remain to be covered that meet $A_0$ and
 two further $A_i,A_j$ ($1\leq i<j\leq u$).
For any fixed pair $i,j$ those triplets form a copy of $\ktri{9,9,3}$,
 thus they are decomposable into 9-cycles by Lemma \ref{l:9+9c9}.

\medskip

\noindent
{\sl Case 3:\/ $v\equiv 6$ (mod 9).}

If $v=9u+6$, beside the sets $A_1,\dots,A_u$ with $|A_i|=9$ we
 take an $A_0$ with $|A_0|=6$.
In this case $A_1$ will be treated differently from the other parts
 $A_i$, $i\geq 2$.
Then we take:
 \begin{itemize}
   \item $\kvv{15}$ inside $A_0\cup A_1$;
   \item $\kvv{9}$ inside each $A_i$ for $2\leq i\leq u$;
   \item $\crt{6}{9}$ between $A_0$ and each $A_i$ for $2\leq i\leq u$;
   \item $\ktri{6,9,9}$ with vertex classes $A_0,A_i,A_j$ for all
     $1\leq i<j\leq u$;
   \item $\crt{9}{9}$ with vertex classes $A_i,A_j$ for all
     $1\leq i<j\leq u$;
   \item $\ktri{9,9,9}$ with vertex classes $A_i,A_j,A_k$ for all
     $1\leq i<j<k\leq u$.
 \end{itemize}
The decompositions of these parts can be done according to the
 lemmas above, and they together decompose $\kvv{9u+6}$
 into 9-cycles.
\qed

\end{Proofof}


\section{2-split systems}
\label{sec:2_split_sys}

A \emph{$2$-split system} \cite{GMT_2020} is a system composed from two vertex-disjoint systems of order $v/2$ 
and a decomposition of the set of edges meeting both parts. In our context a 2-split system requires decompositions of
 $\kvv{v/2}$ and $\crt{v/2}{v/2}$.
In this section we prove that this can be done for all feasible
 residue classes; i.e., 2-split $\CCC{6}{v}$ and $\CCC{9}{v}$
 systems exist for all $v$ that admit $\CCC{6}{v/2}$ and
  $\CCC{9}{v/2}$ systems, respectively.
In fact the latter already follows from our previous lemmas.

\begin{Theorem}
For every\/ $v\equiv 0~(\mathrm{mod}~6)$, $v\geq 18$, there exists a\/ $2$-split\/ $9$-cycle decomposition of\/ $\kvi$.
\end{Theorem}

\begin{proof}
A $\CCC{9}{v/2}$ system with $v/2=3p$ exists for every $p\ge 3$
 by Theorem \ref{t:c9}, and a $\crt{3p}{3p}$ system exists
 by part $(iii)$ of Lemma \ref{l:9+9c9}.
\end{proof}

The construction of 2-split 6-cycle systems requires more work.
Along the way we shall also need a fundamental result on
 Kirkman triple systems.
Recall from the literature that a \emph{Kirkman triple system}
 of order $v$ is a collection of 3-element sets (blocks) such
 that each pair of elements belongs to exactly one block, and
 the set of block can be partitioned into $(v-1)/2$ so-called
  parallel classes, each of those classes consisting of $v/3$
 mutually disjoint blocks.

\begin{Theorem}[Ray-Chaudhuri, Wilson \cite{RCW_1971}]
\label{t:KTS} 
 For every\/ $v\equiv 3~(\mathrm{mod}~6)$ there exists a
 Kirkman triple system of order\/ $v$.
\end{Theorem}

We shall also use the following small construction.

\begin{Lemma}[Akin et al.~\cite{B-etal_2021}, Example 3]
\label{l:66c6} 
There is a decomposition of\/ $\crt{6}{6}$ into\/ $6$-cycles.
\end{Lemma}

\begin{proof}
The construction in \cite{B-etal_2021} takes $\mathbb{Z}_{12}$
 as vertex set, and defines 30 cycles derived from two cycles
 $(0, 5, 10, 8, 11, 2)$ and $(0, 1, 9, 4, 3, 7)$ turning them
 into 12 positions via the mappings $j \mapsto j+i$ (mod 12)
 for $i\in \mathbb{Z}_{12}$, and from a third cycle
 $(0, 1, 2, 6, 7, 8)$ turned into 6 positions via
 $j \mapsto j+i$ (mod 12) for $i=0,1,\dots,5$.
In our notation this corresponds to $a_i=2i-2$ and $b_i=2i-1$
 for $i=1,2,\dots,6$.
\end{proof}

The spectrum of $\CCC{6}{v}$ systems can now be determined.

\begin{Theorem}
 For every\/ $v\equiv 0,6,12~(\mathrm{mod}~24)$, $v\geq 12$, there exists a\/ $2$-split\/ $6$-cycle decomposition of\/ $\kvi$.
\end{Theorem}

\begin{proof}
Consider first the case of residue classes 0 and 12, i.e.\
 where $v$ is of the form $v=12p$.
We know from Theorem \ref{t:c6} that a $\CCC{6}{6p}$ system
 exists for every $p\ge 1$.
Let now $A_1\cup \cdots \cup A_p$ and $B_1\cup \cdots \cup B_p$
 be the vertex sets of two such systems, where the $A_i$ and
 $B_i$ are mutually disjoint 6-element sets.
For all $1\le i,j \le p$ we take decompositions of $\crt{6}{6}$
 as given in Lemma \ref{l:66c6}, for the crossing triplets
 in $A_i\cup B_j$.
It remains to cover the triplets that meet two subsets on
 one side of $\CCC{6}{6p}$ and one subset on the other side.
For any three of those 6-element sets we can apply the
 decomposition chain
 $$
   \ktri{6,6,6} \longrightarrow 2 \times \ktri{3,6,6}
    \longrightarrow 4 \times \ktri{3,3,6}
    \longrightarrow 12 \times \ktri{3,3,2}
 $$
  and find a decomposition according to Lemma \ref{l:tri8}.

For the third residue class $v=24p+6$ the construction starts
 with two copies of a $\CCC{6}{12p+3}$ system, say over the
 disjoint sets $A'$ and $A''$, guaranteed by Theorem \ref{t:c6}.
Apply Theorem \ref{t:KTS} to find Kirkman triple systems over
 $A'$ and $A''$, to be used as auxiliary tools.
We denote by $F'_1,\dots,F'_{6p+1}$ and $F''_1,\dots,F''_{6p+1}$
 the corresponding parallel classes.
Then for $i=1,\dots,6p+1$, for all $T'\in F'_i$ and $T''\in F''_i$
 we take a decomposition of $\kvv{6}$ whose missing edges are
 $T'$ and $T''$ (three 6-cycles for each pair
 $T',T''$) as given in Lemma \ref{l:bi6}.
The collection of those cycles covers all crossing triplets
 exactly once.
Indeed, a triplet $T$ meeting both $A'$ and $A''$ has two vertices on
 one side; those two vertices are contained in a unique block of the
 Kirkman triple system on that side; and the block uniquely determines the index $i$
 of $F'_i \cup F''_i$ in which the block occurs; hence $T$ is
 contained in a single well-defined set $T'\cup T''$ and appears
  in just one of its 6-cycles.
Consequently a $\CCC{6}{24p+6}$ system is obtained.
\end{proof}



\section{Concluding remarks}
\label{sec:concl}

In this note we determined the spectrum of 6-cycle decompositions
 and 9-cycle decompositions of nearly complete 3-uniform hypergraphs
 $\kvi$.
This problem is now completely solved.
On the other hand it would be of interest to find decompositions
 satisfying some further structural requirements. 
Besides the 2-split systems discussed in Section \ref{sec:2_split_sys} we would like to mention cyclic systems as well. 
In a cyclic system, with the vertex set $\zzv$,  the mapping $i\mapsto i+1$ (mod~$v$) is an automorphism of the decomposition.
The missing triplets are $(i,i+v/3,i+2v/3)$,  because $i\mapsto i+1$ has to be an automorphism of the complement of $\kvi$, too.
Concerning cyclic cyctems we formulate the following open problems.

\begin{Conjecture} ~~~
\label{conj}
 \begin{itemize}
  \item[$(i)$] For every\/ $v\equiv 0,3,6~(\mathrm{mod}~12)$, $v\geq 6$,
   there exists a cyclic\/ $6$-cycle decomposition of\/ $\kvi$.
  \item[$(ii)$] For every\/ $v\equiv 0,6,12~(\mathrm{mod}~24)$, $v\geq 12$,
   there exists a cyclic\/ $2$-split\/ $6$-cycle decomposition of\/ $\kvi$.
  \item[$(iii)$] For every\/ $v\equiv 0~(\mathrm{mod}~3)$, $v\geq 9$,
   there exists a cyclic\/ $9$-cycle decomposition of\/ $\kvi$.
  \item[$(iv)$] For every\/ $v\equiv 0~(\mathrm{mod}~6)$, $v\geq 18$,
   there exists a cyclic\/ $2$-split\/ $9$-cycle decomposition of\/ $\kvi$.
 \end{itemize}
\end{Conjecture}

Examples of cyclic 9-cycle decompositions have already been given in
 Lemmas \ref{l:k9c9}, \ref{l:333c9}, and \ref{l:12c9} above.
In further support of Conjecture \ref{conj}, in the two
 subsections below we list the base cycles generating cyclic systems
  $\CCC{6}{v}$ and $\CCC{9}{v}$ for all feasible values of $v\le 30$.
Those systems have been found via a combination of intuitively
 pre-defined base cycles and partial computer search.
In order to facilitate checking that those are decompositions
 indeed, detailed tables are presented in
  Appendices A and C for 6-cycles and
   Appendices B and D for 9-cycles.

\begin{Remark}
Under the mapping\/ $i\mapsto i+1$ the number of orbits of
 triplets other than\/ $(i,i+v/3,i+2v/3)$ is\/ $\frac{1}{6}v(v-3)$.
This is not divisible by\/ $6$ if\/ $v=12p+6$, and not divisible
 by\/ $9$ if\/ $v=9p+6$.
In those residue classes, one ``exceptional'' base cycle has an
 automorphism\/ $i\mapsto i+v/2$ for\/ $6$-cycles and\/
 $i\mapsto i+v/3$ for\/ $9$-cycles.
\end{Remark}

\subsection{Base cycles of cyclic 6-cycle decompositions}
\label{ss:cikl6}

\medskip The cyclic system $\CCC{6}{6}$ has 1 base cycle: \\
(0,1,4,5,2,3) (exceptional).\\

\noindent\begin{tabular}{@{}llll}
\multicolumn{4}{l}{The cyclic system $\CCC{6}{12}$ has 3 base cycles:} \\
(0,1,2,4,5,8),&(0,1,5,8,3,6),&(0,1,9,3,5,7).\\
\end{tabular}
\medskip

\noindent\begin{tabular}{@{}llll}
\multicolumn{4}{l}{The cyclic system $\CCC{6}{15}$  has 5 base cycles:} \\
(0,1,2,4,5,8),& (0,1,5,3,8,6),& (0,1,7,3,6,10),& (0,1,9,14,7,11),\\ 
(0,3,12,6,2,8). \\
\end{tabular}
\medskip

\noindent\begin{tabular}{@{}llll}
\multicolumn{4}{l}{The cyclic system $\CCC{6}{18}$ has 8 base cycles:} \\
(0,1,2,9,10,11)& (exceptional),\\
(0,1,3,4,7,9),& (0,1,5,3,8,10),& (0,1,6,5,2,12),& (0,1,7,3,9,13),\\
(0,2,7,13,3,8),& (0,3,7,17,4,11),& (0,3,12,8,15,6).\\ 
\end{tabular}
\medskip

\noindent\begin{tabular}{@{}lllll}
\multicolumn{4}{l}{The cyclic system $\CCC{6}{24}$ has 14 base cycles:} \\
(0,1,2,4,5,8),& (0,1,5,3,8,6),& (0,1,7,3,6,9),& (0,1,10,3,5,11),\\ 
(0,1,12,3,6,13),& (0,1,14,3,6,15),& (0,1,16,3,6,17),& (0,1,18,3,6,20),\\
(0,2,10,5,1,19),& (0,2,12,7,16,20),&(0,2,14,7,3,16),& (0,4,12,17,3,7),\\ 
(0,4,17,23,7,16),& (0,5,13,19,7,14). \\
\end{tabular}
\medskip

\noindent\begin{tabular}{@{}lllll}
\multicolumn{4}{l}{The cyclic system $\CCC{6}{27}$ has 18 base cycles:} \\
(0,1,3,5,24,26),& (0,3,4,8,23,24),& (0,4,6,1,21,23),& (0,5,6,12,21,22), \\
(0,6,9,1,18,21),& (0,7,3,8,24,20),& (0,8,1,10,26,19),& (0,10,2,11,25,17), \\ 
(0,11,1,8,26,16),& (0,12,2,16,25,15),& (0,13,1,4,26,14),& (0,1,9,4,13,16), \\
(0,2,13,5,18,22),& (0,2,16,23,5,18),& (0,3,16,6,21,17),& (0,4,12,19,7,21), \\
(0,6,17,23,12,19),& (0,6,18,4,23,7). \\
\end{tabular}
\medskip

\noindent\begin{tabular}{@{}lllll}
\multicolumn{4}{l}{The cyclic system $\CCC{6}{30}$ has 23 base cycles:} \\
(0,1,2,15,16,17)& (exceptional), \\
(0,2,3,15,27,28),& (0,3,4,15,26,27),& (0,4,6,15,24,26),& (0,5,1,15,29,25), \\
(0,6,7,15,23,24),& (0,7,2,15,28,23),& (0,1,6,3,10,15),& (0,1,8,10,2,20), \\
(0,2,10,5,14,17),& (0,2,13,5,19,14),& (0,3,9,16,1,10),& (0,4,10,29,13,23), \\
(0,4,12,25,7,21),& (0,10,21,6,24,13),& (0,5,12,9,14,15),& (0,18,10,12,19,20), \\
(0,3,12,7,15,17),& (0,25,9,1,12,14),& (0,9,24,1,7,10),& (0,10,24,13,19,23), \\
(0,14,26,9,17,21),& (0,19,7,22,3,13). \\
\end{tabular}
\medskip

\subsection{Base cycles of cyclic 9-cycle decompositions}
\label{ss:cikl9}

\noindent The cyclic system $\CCC{9}{9}$ has 1 base cycle: \\
(0,1,2,7,6,4,8,5,3). (Cf.\ Lemma 7.) \\

\noindent\begin{tabular}{@{}lll}
\multicolumn{3}{l}{The cyclic system $\CCC{9}{12}$ has 2 base cycles:} \\
(0,1,2,4,5,8,10,3,6),& (0,1,7,11,2,9,10,6,8). \\
\end{tabular}
\medskip

\noindent\begin{tabular}{@{}lll}
\multicolumn{3}{l}{The cyclic system $\CCC{9}{15}$ has 4 base cycles:} \\
(0,1,2,5,6,7,10,11,12)& (exceptional), \\
(0,1,3,4,8,2,5,7,9),& (0,1,6,3,10,2,12,4,7),& (0,1,10,14,6,12,2,5,11).\\
\end{tabular}
\medskip

\noindent\begin{tabular}{@{}lll}
\multicolumn{3}{l}{The cyclic system $\CCC{9}{18}$ has 5 base cycles:} \\
(0,1,2,4,5,8,3,6,7),& (0,1,5,3,9,4,2,10,8),& (0,1,9,4,8,10,3,6,13), \\
(0,1,10,3,7,14,17,2,11),& (0,3,13,17,5,10,14,4,9). \\
\end{tabular}
\medskip

\noindent\begin{tabular}{@{}lllll}
\multicolumn{3}{l}{The cyclic system $\CCC{9}{21}$ has 7 base cycles:} \\
(0,1,2,4,5,8,3,6,7),& (0,1,5,3,9,4,2,10,8),& (0,1,9,4,8,10,3,6,12), \\
(0,1,10,3,7,11,2,6,13),& (0,1,11,3,8,13,2,4,14),& (0,3,6,18,4,13,1,5,14), \\
(0,3,13,7,12,20,6,17,11).\\
\end{tabular}
\medskip

\noindent\begin{tabular}{@{}lllll}
\multicolumn{3}{l}{The cyclic system $\CCC{9}{24}$ has 10 base cycles:} \\
(0,1,2,8,9,10,16,17,18)& (exceptional), \\
(0,1,3,4,7,2,5,6,10),& (0,1,6,5,13,2,3,11,12),& (0,1,14,3,5,7,11,2,15), \\
(0,1,17,3,5,10,2,6,19),& (0,2,8,5,1,7,9,16,12),& (0,2,10,7,1,12,5,14,19), \\
(0,3,9,15,1,4,14,20,10),& (0,3,12,15,4,22,13,5,17),& (0,4,9,16,1,11,18,5,13). \\
\end{tabular}
\medskip

\noindent\begin{tabular}{@{}lllll}
\multicolumn{3}{l}{The cyclic system $\CCC{9}{27}$ has 12 base cycles:} \\
(0,1,3,6,10,17,21,24,26),& (0,2,6,1,7,20,26,21,25),& (0,3,4,10,1,26,17,23,24), \\
(0,4,9,1,8,19,26,18,23),& (0,5,7,13,2,25,14,20,22),& (0,6,10,18,1,26,9,17,21), \\
(0,7,9,19,5,22,8,18,20),& (0,8,9,18,4,23,9,18,19),& (0,10,3,15,1,26,12,24,17), \\
(0,11,1,22,8,19,5,26,16),& (0,12,9,1,16,11,26,18,15),& (0,13,4,19,20,7,8,23,14). \\
\end{tabular}
\medskip

\noindent\begin{tabular}{@{}lllll}
\multicolumn{3}{l}{The cyclic system $\CCC{9}{30}$ has 15 base cycles:} \\
(0,1,3,6,7,23,24,27,29),& (0,2,6,1,7,23,29,24,28),& (0,3,7,12,1,29,18,23,27), \\
(0,4,10,1,8,22,29,20,26),& (0,5,7,13,6,24,17,23,25),& (0,6,13,1,4,26,29,17,24), \\
(0,7,8,16,1,29,14,22,23),& (0,8,10,19,2,28,11,20,22),& (0,9,10,20,2,28,10,20,21), \\
(0,11,3,16,28,2,14,27,19),& (0,12,1,5,20,10,25,29,18),& (0,13,1,11,8,22,19,29,17), \\
(0,14,2,17,3,27,13,28,16),& (0,3,13,22,8,18,21,5,16),& (0,6,15,25,8,18,4,13,19). \\
\end{tabular}

\newpage

\paragraph{Acknowledgements.}
This research was supported by the National Research,
 Development, and Innovation Office -- NKFIH
 under the grant SNN 129364, and by the
  Artificial Intelligence National Laboratory (MILAB).


\newpage


\section*{Appendices --- Data of base cycles in cyclic systems}

Here we provide additional data for the cyclic systems
 described in Sections \ref{ss:cikl6} and \ref{ss:cikl9},
 arranged in four parts A, B, C, D.
The parts A and C deal with 6-cycle systems, while B and D
 deal with 9-cycle systems.
The tables in A and B are simpler, they just list what kind of
 triplets the corresponding base cycles cover.
Under the assumption that those data are correct, they are
 sufficient to check that all triplets of that $\kvi$ are covered.
Further details are given in C and D; they help to verify that
 the covered triplets are indeed those claimed in A and B.

Recall that the vertex set of a cyclic system of order $v$ is $\zzv$.
Let $(a,b,c)$ be any vertex triplet, where $0\le a<b<c\le v-1$.
Under the automorphism $i\mapsto i+1$ (mod $v$), the orbit of
 $(a,b,c)$ consists of $v$ triplets
 (except for triplets of the form $(a,a+v/3,a+2v/3)$, which are
 not members of $\kvi$ if just cyclic decompositions are wanted).
Three members of the orbit contain vertex 1.
Among the three we choose the one that is lexicographically
 smallest, and call it the \emph{type} of $(a,b,c)$.

A collection of base cycles generates a cyclic decomposition
 of $\kvi$ if and only if it covers all types.
Appendices A and B repeat the base cycles given in Sections
 \ref{ss:cikl6} and \ref{ss:cikl9}, and list the types
 covered by them.
Each base cycle contains vertex 0 as its 1st vertex, and its
 other (2nd, 3rd, etc.) vertices are listed in the order as
 they appear along the cycle in question.
In the same row of the table, the type of any three consecutive
 vertices is given.

The tables in Appendices C and D include additional columns
 marked with ``d''.
The values in those cells tell the ``smallest distance'',
 $b'-a'$, where $(a,b,c)$ is of type $(a',b',c')$
 (in particular, $a'=1$ and the inequalities $b'-a'\le c'-b'$
  and $b'-a'\le \min \{ c'-a', v-c'+a' \}$ are valid,
  by the definition of ``type'').
The triplet types covered by the base cycles are listed in
 subtables as ``first three'' and ``last three'' for
 6-cycles, and the same with ``next three'' included in the
 middle for 9-cycles.

\newpage
\section*{Appendix A: Cyclic systems --- \hfill\break
\hspace*{5.7em} 6-cycle decomposition}
\label{App:C6}

\begin{table}[htbp]
\centering
\caption{Base cycles of system $\CCC{6}{12}$ }
\begin{tabular}{|c|c|c|c|c|c||c|c|c|c|c|c|}
\hline
\multicolumn{ 6}{|c||}{Vertices} & \multicolumn{ 6}{c|}{Triplet types} \\ \hline
\textbf{1} & \textbf{2} & \textbf{3} & \textbf{4} & \textbf{5} & \textbf{6} &  \textbf{1-2-3} &  \textbf{2-3-4} &  \textbf{3-4-5} &  \textbf{4-5-6} &  \textbf{5-6-1} &  \textbf{6-1-2} \\ \hline
\hline
0 & 1 & 2 & 4 & 5 & 8 & 1 2 3 & 1 2 4 & 1 2 11 & 1 2 5 & 1 4 8 & 1 2 9 \\ \hline
0 & 1 & 5 & 8 & 3 & 6 & 1 2 6 & 1 4 9 & 1 3 6 & 1 3 10 & 1 4 7 & 1 2 7 \\ \hline
0 & 1 & 9 & 3 & 5 & 7 & 1 2 10 & 1 3 9 & 1 3 7 & 1 3 5 & 1 3 8 & 1 2 8 \\ \hline
\end{tabular}
\label{t:C6:12}
\end{table}

\begin{table}[htbp]
\centering
\caption{Base cycles of system  $\CCC{6}{15}$ }
\begin{tabular}{|c|c|c|c|c|c||c|c|c|c|c|c|}
\hline
\multicolumn{ 6}{|c||}{Vertices} & \multicolumn{ 6}{c|}{Triplet types} \\ \hline
\textbf{1} & \textbf{2} & \textbf{3} & \textbf{4} & \textbf{5} & \textbf{6} &  \textbf{1-2-3} &  \textbf{2-3-4} &  \textbf{3-4-5} &  \textbf{4-5-6} &  \textbf{5-6-1} &  \textbf{6-1-2} \\ \hline
\hline
0 & 1 & 2 & 4 & 5 & 8 & 1 2 3 & 1 2 4 & 1 2 14 & 1 2 5 & 1 4 11 & 1 2 9 \\ \hline
0 & 1 & 5 & 3 & 8 & 6 & 1 2 6 & 1 3 5 & 1 3 6 & 1 3 13 & 1 3 10 & 1 2 7 \\ \hline
0 & 1 & 7 & 3 & 6 & 10 & 1 2 8 & 1 3 7 & 1 2 13 & 1 4 8 & 1 5 10 & 1 2 11 \\ \hline
0 & 1 & 9 & 14 & 7 & 11 & 1 2 10 & 1 3 11 & 1 3 8 & 1 4 12 & 1 5 9 & 1 2 12 \\ \hline
0 & 3 & 12 & 6 & 2 & 8 & 1 4 7 & 1 4 10 & 1 5 11 & 1 3 12 & 1 3 9 & 1 4 9 \\ \hline
\end{tabular}
\label{t:C6:15}
\end{table}

\begin{table}[htbp]
\centering
\caption{Base cycles of system  $\CCC{6}{18}$ }
\begin{tabular}{|c|c|c|c|c|c||c|c|c|c|c|c|}
\hline
\multicolumn{ 6}{|c||}{Vertices} & \multicolumn{ 6}{c|}{Triplet types} \\ \hline
\textbf{1} & \textbf{2} & \textbf{3} & \textbf{4} & \textbf{5} & \textbf{6} &  \textbf{1-2-3} &  \textbf{2-3-4} &  \textbf{3-4-5} &  \textbf{4-5-6} &  \textbf{5-6-1} &  \textbf{6-1-2} \\ \hline
\hline
\multicolumn{12}{|l|}{Base cycle generating 9 cycles:} \\ \hline
0 & 1 & 2 & 9 & 10 & 11 & 1 2 3 & 1 2 9 & 1 2 12 & 1 2 3 & 1 2 9 & 1 2 12 \\ \hline
\multicolumn{12}{|l|}{Base cycles generating 18 cycles:} \\ \hline
0 & 1 & 3 & 4 & 7 & 9 & 1 2 4 & 1 2 17 & 1 2 5 & 1 3 16 & 1 3 12 & 1 2 10 \\ \hline
0 & 1 & 5 & 3 & 8 & 10 & 1 2 6 & 1 3 5 & 1 3 6 & 1 3 14 & 1 3 11 & 1 2 11 \\ \hline
0 & 1 & 6 & 5 & 2 & 12 & 1 2 7 & 1 2 15 & 1 2 16 & 1 4 11 & 1 3 13 & 1 2 13 \\ \hline
0 & 1 & 7 & 3 & 9 & 13 & 1 2 8 & 1 3 7 & 1 3 15 & 1 5 13 & 1 5 10 & 1 2 14 \\ \hline
0 & 2 & 7 & 13 & 3 & 8 & 1 3 8 & 1 6 12 & 1 5 11 & 1 6 11 & 1 4 9 & 1 3 9 \\ \hline
0 & 3 & 7 & 17 & 4 & 11 & 1 4 8 & 1 5 9 & 1 4 14 & 1 6 13 & 1 5 12 & 1 4 12 \\ \hline
0 & 3 & 12 & 8 & 15 & 6 & 1 4 13 & 1 5 14 & 1 4 15 & 1 3 10 & 1 4 10 & 1 4 7 \\ \hline
\end{tabular}
\label{t:C6:18}
\end{table}

\begin{table}[htbp]
\centering
\caption{Base cycles of system  $\CCC{6}{24}$ }
\begin{tabular}{|c|c|c|c|c|c||c|c|c|c|c|c|}
\hline
\multicolumn{ 6}{|c||}{Vertices} & \multicolumn{ 6}{c|}{Triplet types} \\ \hline
\textbf{1} & \textbf{2} & \textbf{3} & \textbf{4} & \textbf{5} & \textbf{6} &  \textbf{1-2-3} &  \textbf{2-3-4} &  \textbf{3-4-5} &  \textbf{4-5-6} &  \textbf{5-6-1} &  \textbf{6-1-2} \\ \hline
\hline
0 & 1 & 2 & 4 & 5 & 8 & 1 2 3 & 1 2 4 & 1 2 23 & 1 2 5 & 1 4 20 & 1 2 9 \\ \hline
0 & 1 & 5 & 3 & 8 & 6 & 1 2 6 & 1 3 5 & 1 3 6 & 1 3 22 & 1 3 19 & 1 2 7 \\ \hline
0 & 1 & 7 & 3 & 6 & 9 & 1 2 8 & 1 3 7 & 1 2 22 & 1 4 7 & 1 4 19 & 1 2 10 \\ \hline
0 & 1 & 10 & 3 & 5 & 11 & 1 2 11 & 1 3 10 & 1 3 8 & 1 3 9 & 1 6 12 & 1 2 12 \\ \hline
0 & 1 & 12 & 3 & 6 & 13 & 1 2 13 & 1 3 12 & 1 4 10 & 1 4 11 & 1 7 14 & 1 2 14 \\ \hline
0 & 1 & 14 & 3 & 6 & 15 & 1 2 15 & 1 3 14 & 1 4 12 & 1 4 13 & 1 7 16 & 1 2 16 \\ \hline
0 & 1 & 16 & 3 & 6 & 17 & 1 2 17 & 1 3 16 & 1 4 14 & 1 4 15 & 1 7 18 & 1 2 18 \\ \hline
0 & 1 & 18 & 3 & 6 & 20 & 1 2 19 & 1 3 18 & 1 4 16 & 1 4 18 & 1 5 11 & 1 2 21 \\ \hline
0 & 2 & 10 & 5 & 1 & 19 & 1 3 11 & 1 4 9 & 1 5 10 & 1 5 19 & 1 2 20 & 1 3 20 \\ \hline
0 & 2 & 12 & 7 & 16 & 20 & 1 3 13 & 1 6 11 & 1 5 20 & 1 5 16 & 1 5 9 & 1 3 21 \\ \hline
0 & 2 & 14 & 7 & 3 & 16 & 1 3 15 & 1 6 13 & 1 5 12 & 1 5 14 & 1 4 17 & 1 3 17 \\ \hline
0 & 4 & 12 & 17 & 3 & 7 & 1 5 13 & 1 6 17 & 1 6 16 & 1 5 15 & 1 4 8 & 1 4 21 \\ \hline
0 & 4 & 17 & 23 & 7 & 16 & 1 5 18 & 1 6 19 & 1 7 15 & 1 8 16 & 1 8 17 & 1 5 17 \\ \hline
0 & 5 & 13 & 19 & 7 & 14 & 1 6 14 & 1 7 17 & 1 7 13 & 1 6 18 & 1 8 15 & 1 6 15 \\ \hline
\end{tabular}
\label{t:C6:24}
\end{table}

\begin{table}[htbp]
\centering
\caption{Base cycles of system $\CCC{6}{27}$ }
\begin{tabular}{|c|c|c|c|c|c||c|c|c|c|c|c|}
\hline
\multicolumn{ 6}{|c||}{Vertices} & \multicolumn{ 6}{c|}{Triplet types} \\ \hline
\textbf{1} & \textbf{2} & \textbf{3} & \textbf{4} & \textbf{5} & \textbf{6} &  \textbf{1-2-3} &  \textbf{2-3-4} &  \textbf{3-4-5} &  \textbf{4-5-6} &  \textbf{5-6-1} &  \textbf{6-1-2} \\ \hline
\hline
0 & 1 & 3 & 5 & 24 & 26 & 1 2 4 & 1 3 5 & 1 3 22 & 1 3 9 & 1 2 26 & 1 2 3 \\ \hline
0 & 3 & 4 & 8 & 23 & 24 & 1 2 25 & 1 2 6 & 1 5 20 & 1 2 13 & 1 2 5 & 1 4 7 \\ \hline
0 & 4 & 6 & 1 & 21 & 23 & 1 3 24 & 1 3 25 & 1 6 21 & 1 3 8 & 1 3 7 & 1 5 9 \\ \hline
0 & 5 & 6 & 12 & 21 & 22 & 1 2 23 & 1 2 8 & 1 7 16 & 1 2 19 & 1 2 7 & 1 6 11 \\ \hline
0 & 6 & 9 & 1 & 18 & 21 & 1 4 22 & 1 4 23 & 1 9 18 & 1 4 11 & 1 4 10 & 1 7 13 \\ \hline
0 & 7 & 3 & 8 & 24 & 20 & 1 4 8 & 1 2 24 & 1 6 22 & 1 5 16 & 1 4 24 & 1 8 15 \\ \hline
0 & 8 & 1 & 10 & 26 & 19 & 1 2 9 & 1 3 21 & 1 3 12 & 1 8 19 & 1 2 21 & 1 9 17 \\ \hline
0 & 10 & 2 & 11 & 25 & 17 & 1 3 11 & 1 2 20 & 1 5 14 & 1 7 15 & 1 3 20 & 1 8 18 \\ \hline
0 & 11 & 1 & 8 & 26 & 16 & 1 2 12 & 1 4 21 & 1 3 10 & 1 9 19 & 1 2 18 & 1 6 17 \\ \hline
0 & 12 & 2 & 16 & 25 & 15 & 1 3 13 & 1 5 18 & 1 5 19 & 1 2 11 & 1 3 18 & 1 4 16 \\ \hline
0 & 13 & 1 & 4 & 26 & 14 & 1 2 14 & 1 4 13 & 1 3 6 & 1 6 16 & 1 2 16 & 1 2 15 \\ \hline
0 & 1 & 9 & 4 & 13 & 16 & 1 2 10 & 1 4 9 & 1 5 23 & 1 4 19 & 1 4 15 & 1 2 17 \\ \hline
0 & 2 & 13 & 5 & 18 & 22 & 1 3 14 & 1 4 12 & 1 6 20 & 1 5 15 & 1 5 10 & 1 3 23 \\ \hline
0 & 2 & 16 & 23 & 5 & 18 & 1 3 17 & 1 7 21 & 1 8 17 & 1 6 15 & 1 6 19 & 1 3 19 \\ \hline
0 & 3 & 16 & 6 & 21 & 17 & 1 4 17 & 1 4 14 & 1 6 18 & 1 5 17 & 1 5 11 & 1 4 18 \\ \hline
0 & 4 & 12 & 19 & 7 & 21 & 1 5 13 & 1 8 20 & 1 6 13 & 1 3 16 & 1 7 14 & 1 5 22 \\ \hline
0 & 6 & 17 & 23 & 12 & 19 & 1 7 18 & 1 7 17 & 1 6 12 & 1 5 21 & 1 8 16 & 1 7 20 \\ \hline
0 & 6 & 18 & 4 & 23 & 7 & 1 7 19 & 1 3 15 & 1 6 14 & 1 4 20 & 1 5 12 & 1 2 22 \\ \hline
\end{tabular}
\label{t:C6:27}
\end{table}

\begin{table}[htbp]
\centering
\caption{Base cycles of system  $\CCC{6}{30}$ }
\begin{tabular}{|c|c|c|c|c|c||c|c|c|c|c|c|}
\hline
\multicolumn{ 6}{|c||}{Vertices} & \multicolumn{ 6}{c|}{Triplet types} \\ \hline
\textbf{1} & \textbf{2} & \textbf{3} & \textbf{4} & \textbf{5} & \textbf{6} &  \textbf{1-2-3} &  \textbf{2-3-4} &  \textbf{3-4-5} &  \textbf{4-5-6} &  \textbf{5-6-1} &  \textbf{6-1-2} \\ \hline
\hline
\multicolumn{12}{|l|}{Base cycle generating 15 cycles:} \\ \hline
0 & 1 & 2 & 15 & 16 & 17 & 1 2 3 & 1 2 15 & 1 2 18 & 1 2 3 & 1 2 15 & 1 2 18 \\ \hline
\multicolumn{12}{|l|}{Base cycles generating 30 cycles:} \\ \hline
0 & 2 & 3 & 15 & 27 & 28 & 1 2 29 & 1 2 14 & 1 7 19 & 1 2 19 & 1 2 4 & 1 3 5 \\ \hline
0 & 3 & 4 & 15 & 26 & 27 & 1 2 28 & 1 2 13 & 1 9 20 & 1 2 20 & 1 2 5 & 1 4 7 \\ \hline
0 & 4 & 6 & 15 & 24 & 26 & 1 3 27 & 1 3 12 & 1 10 19 & 1 3 22 & 1 3 7 & 1 5 9 \\
\hline

\end{tabular}
\label{t:C6:30}
\end{table}


\newpage

\section*{Appendix B: Cyclic systems --- \hfill\break
\hspace*{5.7em} 9-cycle decomposition}
\label{App:C9}

\begin{table}[htbp]
\centering
\small
\caption{Base cycle of system  $\CCC{9}{9}$ }
\begin{tabular}{|c|c|c|c|c|c|c|c|c||c|c|c|c|c|c|c|c|c|}
\hline
\multicolumn{ 9}{|c||}{Vertices} & \multicolumn{ 9}{c|}{Triplet types} \\ \hline
\textbf{1} & \textbf{2} & \textbf{3} & \textbf{4} & \textbf{5} & \textbf{6} & \textbf{7} & \textbf{8} & \textbf{9} & \textbf{1-2-3} & \textbf{2-3-4} & \textbf{3-4-5} & \textbf{4-5-6} & \textbf{5-6-7} & \textbf{6-7-8} & \textbf{7-8-9} & \textbf{8-9-1} & \textbf{9-1-2} \\ \hline
\hline
0 & 1 & 2 & 7 & 6 & 4 & 8 & 5 & 3 & 1 2 3 & 1 2 7 & 1 2 6 & 1 2 8 & 1 3 5 & 1 2 5 & 1 3 6 & 1 3 7 & 1 2 4 \\ \hline
\end{tabular}
\label{t:C9:9}
\end{table}

\begin{table}[htbp]
\centering
\small
\caption{Base cycles of system $\CCC{9}{12}$  }
\begin{tabular}{|c|c|c|c|c|c|c|c|c||c|c|c|c|c|c|c|c|c|}
\hline
\multicolumn{ 9}{|c||}{Vertices} & \multicolumn{ 9}{c|}{Triplet types} \\ \hline
\textbf{1} & \textbf{2} & \textbf{3} & \textbf{4} & \textbf{5} & \textbf{6} & \textbf{7} & \textbf{8} & \textbf{9} & \textbf{1-2-3} & \textbf{2-3-4} & \textbf{3-4-5} & \textbf{4-5-6} & \textbf{5-6-7} & \textbf{6-7-8} & \textbf{7-8-9} & \textbf{8-9-1} & \textbf{9-1-2} \\ \hline
\hline
0 & 1 & 2 & 4 & 5 & 8 & 10 & 3 & 6 & 1 2 3 & 1 2 4 & 1 2 11 & 1 2 5 & 1 3 10 & 1 3 8 & 1 4 8 & 1 4 10 & 1 2 7 \\ \hline
0 & 1 & 7 & 11 & 2 & 9 & 10 & 6 & 8 & 1 2 8 & 1 3 9 & 1 4 9 & 1 3 6 & 1 2 6 & 1 2 10 & 1 3 5 & 1 3 7 & 1 2 9 \\ \hline
\end{tabular}
\label{t:C9:12}
\end{table}

\begin{table}[htbp]
\centering
\small
\caption{Base cycles of system $\CCC{9}{15}$ }
\begin{tabular}{|c|c|c|c|c|c|c|c|c||c|c|c|c|c|c|c|c|c|}
\hline
\multicolumn{ 9}{|c||}{Vertices} & \multicolumn{ 9}{c|}{Triplet types} \\ \hline
\textbf{1} & \textbf{2} & \textbf{3} & \textbf{4} & \textbf{5} & \textbf{6} & \textbf{7} & \textbf{8} & \textbf{9} & \textbf{1-2-3} & \textbf{2-3-4} & \textbf{3-4-5} & \textbf{4-5-6} & \textbf{5-6-7} & \textbf{6-7-8} & \textbf{7-8-9} & \textbf{8-9-1} & \textbf{9-1-2} \\ \hline
\hline
\multicolumn{18}{|l|}{Base cycle generating 5 cycles:} \\ \hline
0 & 1 & 2 & 5 & 6 & 7 & 10 & 11 & 12 & 1 2 3 & 1 2 5 & 1 2 13 & 1 2 3 & 1 2 5 & 1 2 13 & 1 2 3 & 1 2 5 & 1 2 13 \\ \hline
\multicolumn{18}{|l|}{Base cycles generating 15 cycles:} \\ \hline
0 & 1 & 3 & 4 & 8 & 2 & 5 & 7 & 9 & 1 2 4 & 1 2 14 & 1 2 6 & 1 3 7 & 1 4 7 & 1 3 13 & 1 3 5 & 1 3 9 & 1 2 10 \\ \hline
0 & 1 & 6 & 3 & 10 & 2 & 12 & 4 & 7 & 1 2 7 & 1 3 6 & 1 4 8 & 1 2 9 & 1 3 8 & 1 3 11 & 1 4 9 & 1 4 12 & 1 2 8 \\ \hline
0 & 1 & 10 & 14 & 6 & 12 & 2 & 5 & 11 & 1 2 11 & 1 3 12 & 1 5 9 & 1 3 10 & 1 5 11 & 1 4 11 & 1 4 10 & 1 5 10 & 1 2 12 \\ \hline
\end{tabular}
\label{t:C9:15}
\end{table}

\begin{table}[htbp]
\centering
\small
\caption{Base cycles of system$\CCC{9}{18}$ }
\begin{tabular}{|c|c|c|c|c|c|c|c|c||c|c|c|c|c|c|c|c|c|}
\hline
\multicolumn{ 9}{|c||}{Vertices} & \multicolumn{ 9}{c|}{Triplet types} \\ \hline
\textbf{1} & \textbf{2} & \textbf{3} & \textbf{4} & \textbf{5} & \textbf{6} & \textbf{7} & \textbf{8} & \textbf{9} & \textbf{1-2-3} & \textbf{2-3-4} & \textbf{3-4-5} & \textbf{4-5-6} & \textbf{5-6-7} & \textbf{6-7-8} & \textbf{7-8-9} & \textbf{8-9-1} & \textbf{9-1-2} \\ \hline
\hline
0 & 1 & 2 & 4 & 5 & 8 & 3 & 6 & 7 & 1 2 3 & 1 2 4 & 1 2 17 & 1 2 5 & 1 3 6 & 1 3 16 & 1 2 16 & 1 2 13 & 1 2 8 \\ \hline
0 & 1 & 5 & 3 & 9 & 4 & 2 & 10 & 8 & 1 2 6 & 1 3 5 & 1 3 7 & 1 2 7 & 1 3 8 & 1 3 9 & 1 3 13 & 1 3 11 & 1 2 9 \\ \hline
0 & 1 & 9 & 4 & 8 & 10 & 3 & 6 & 13 & 1 2 10 & 1 4 9 & 1 2 15 & 1 3 15 & 1 3 14 & 1 4 8 & 1 4 11 & 1 6 12 & 1 2 14 \\ \hline
0 & 1 & 10 & 3 & 7 & 14 & 17 & 2 & 11 & 1 2 11 & 1 3 10 & 1 4 15 & 1 5 12 & 1 4 12 & 1 4 7 & 1 4 13 & 1 3 12 & 1 2 12 \\ \hline
0 & 3 & 13 & 17 & 5 & 10 & 14 & 4 & 9 & 1 4 14 & 1 5 9 & 1 5 11 & 1 6 13 & 1 5 14 & 1 5 13 & 1 6 11 & 1 5 10 & 1 4 10 \\ \hline
\end{tabular}
\label{t:C9:18}
\end{table}

\begin{table}[htbp]
\centering
\small
\caption{Base cycles of system $\CCC{9}{21}$ }
\begin{tabular}{|c|c|c|c|c|c|c|c|c||c|c|c|c|c|c|c|c|c|}
\hline
\multicolumn{ 9}{|c||}{Vertices} & \multicolumn{ 9}{c|}{Triplet types} \\ \hline
\textbf{1} & \textbf{2} & \textbf{3} & \textbf{4} & \textbf{5} & \textbf{6} & \textbf{7} & \textbf{8} & \textbf{9} & \textbf{1-2-3} & \textbf{2-3-4} & \textbf{3-4-5} & \textbf{4-5-6} & \textbf{5-6-7} & \textbf{6-7-8} & \textbf{7-8-9} & \textbf{8-9-1} & \textbf{9-1-2} \\ \hline
\hline
0 & 1 & 2 & 4 & 5 & 8 & 3 & 6 & 7 & 1 2 3 & 1 2 4 & 1 2 20 & 1 2 5 & 1 3 6 & 1 3 19 & 1 2 19 & 1 2 16 & 1 2 8 \\ \hline
0 & 1 & 5 & 3 & 9 & 4 & 2 & 10 & 8 & 1 2 6 & 1 3 5 & 1 3 7 & 1 2 7 & 1 3 8 & 1 3 9 & 1 3 16 & 1 3 14 & 1 2 9 \\ \hline
0 & 1 & 9 & 4 & 8 & 10 & 3 & 6 & 12 & 1 2 10 & 1 4 9 & 1 2 18 & 1 3 18 & 1 3 17 & 1 4 8 & 1 4 10 & 1 7 13 & 1 2 13 \\ \hline
0 & 1 & 10 & 3 & 7 & 11 & 2 & 6 & 13 & 1 2 11 & 1 3 10 & 1 4 18 & 1 5 9 & 1 5 17 & 1 5 10 & 1 5 12 & 1 7 14 & 1 2 14 \\ \hline
0 & 1 & 11 & 3 & 8 & 13 & 2 & 4 & 14 & 1 2 12 & 1 3 11 & 1 4 17 & 1 6 11 & 1 6 16 & 1 3 12 & 1 3 13 & 1 5 15 & 1 2 15 \\ \hline
0 & 3 & 6 & 18 & 4 & 13 & 1 & 5 & 14 & 1 4 7 & 1 4 16 & 1 3 15 & 1 6 13 & 1 4 13 & 1 5 13 & 1 5 14 & 1 6 15 & 1 4 15 \\ \hline
0 & 3 & 13 & 7 & 12 & 20 & 6 & 17 & 11 & 1 4 14 & 1 5 11 & 1 2 17 & 1 6 14 & 1 7 15 & 1 4 11 & 1 6 12 & 1 5 16 & 1 4 12 \\ \hline
\end{tabular}
\label{t:C9:21}
\end{table}

\begin{table}[htbp]
\centering
\small
\caption{Base cycles of system $\CCC{9}{24}$ }
\begin{tabular}{|c|c|c|c|c|c|c|c|c||c|c|c|c|c|c|c|c|c|}
\hline
\multicolumn{ 9}{|c||}{Vertices} & \multicolumn{ 9}{c|}{Triplet types} \\ \hline
\textbf{1} & \textbf{2} & \textbf{3} & \textbf{4} & \textbf{5} & \textbf{6} & \textbf{7} & \textbf{8} & \textbf{9} & \textbf{1-2-3} & \textbf{2-3-4} & \textbf{3-4-5} & \textbf{4-5-6} & \textbf{5-6-7} & \textbf{6-7-8} & \textbf{7-8-9} & \textbf{8-9-1} & \textbf{9-1-2} \\ \hline
\hline
\multicolumn{18}{|l|}{Base cycle generating 8 cycles:} \\ \hline
0 & 1 & 2 & 8 & 9 & 10 & 16 & 17 & 18 & 1 2 3 & 1 2 8 & 1 2 19 & 1 2 3 & 1 2 8 & 1 2 19 & 1 2 3 & 1 2 8 & 1 2 19 \\ \hline
\multicolumn{18}{|l|}{Base cycles generating 24 cycles:} \\ \hline
0 & 1 & 3 & 4 & 7 & 2 & 5 & 6 & 10 & 1 2 4 & 1 2 23 & 1 2 5 & 1 3 6 & 1 3 22 & 1 2 22 & 1 2 6 & 1 5 19 & 1 2 11 \\ \hline
0 & 1 & 6 & 5 & 13 & 2 & 3 & 11 & 12 & 1 2 7 & 1 2 21 & 1 2 9 & 1 4 12 & 1 2 12 & 1 2 10 & 1 2 17 & 1 2 14 & 1 2 13 \\ \hline
0 & 1 & 14 & 3 & 5 & 7 & 11 & 2 & 15 & 1 2 15 & 1 3 14 & 1 3 12 & 1 3 5 & 1 3 7 & 1 5 20 & 1 5 16 & 1 3 16 & 1 2 16 \\ \hline
0 & 1 & 17 & 3 & 5 & 10 & 2 & 6 & 19 & 1 2 18 & 1 3 17 & 1 3 15 & 1 3 8 & 1 4 9 & 1 5 9 & 1 5 18 & 1 6 12 & 1 2 20 \\ \hline
0 & 2 & 8 & 5 & 1 & 7 & 9 & 16 & 12 & 1 3 9 & 1 4 7 & 1 4 21 & 1 3 21 & 1 3 19 & 1 3 10 & 1 4 8 & 1 5 13 & 1 3 13 \\ \hline
0 & 2 & 10 & 7 & 1 & 12 & 5 & 14 & 19 & 1 3 11 & 1 4 20 & 1 4 19 & 1 6 19 & 1 5 12 & 1 3 18 & 1 6 16 & 1 6 11 & 1 3 20 \\ \hline
0 & 3 & 9 & 15 & 1 & 4 & 14 & 20 & 10 & 1 4 10 & 1 7 13 & 1 7 17 & 1 4 15 & 1 4 14 & 1 7 15 & 1 5 11 & 1 5 15 & 1 4 11 \\ \hline
0 & 3 & 12 & 15 & 4 & 22 & 13 & 5 & 17 & 1 4 13 & 1 4 16 & 1 4 17 & 1 7 18 & 1 7 16 & 1 8 16 & 1 5 17 & 1 6 18 & 1 4 18 \\ \hline
0 & 4 & 9 & 16 & 1 & 11 & 18 & 5 & 13 & 1 5 10 & 1 6 13 & 1 8 17 & 1 6 15 & 1 8 15 & 1 7 14 & 1 6 17 & 1 6 14 & 1 5 14 \\ \hline
\end{tabular}
\label{t:C9:24}
\end{table}

\begin{table}[htbp]
\centering
\small
\caption{Base cycles of system $\CCC{9}{27}$ }
\begin{tabular}{|c|c|c|c|c|c|c|c|c||c|c|c|c|c|c|c|c|c|}
\hline
\multicolumn{ 9}{|c||}{Vertices} & \multicolumn{ 9}{c|}{Triplet types} \\ \hline
\textbf{1} & \textbf{2} & \textbf{3} & \textbf{4} & \textbf{5} & \textbf{6} & \textbf{7} & \textbf{8} & \textbf{9} & \textbf{1-2-3} & \textbf{2-3-4} & \textbf{3-4-5} & \textbf{4-5-6} & \textbf{5-6-7} & \textbf{6-7-8} & \textbf{7-8-9} & \textbf{8-9-1} & \textbf{9-1-2} \\ \hline
\hline
0 & 1 & 3 & 6 & 10 & 17 & 21 & 24 & 26 & 1 2 4 & 1 3 6 & 1 4 8 & 1 5 12 & 1 5 21 & 1 4 24 & 1 3 25 & 1 2 26 & 1 2 3 \\ \hline
0 & 2 & 6 & 1 & 7 & 20 & 26 & 21 & 25 & 1 3 7 & 1 2 6 & 1 2 23 & 1 7 20 & 1 7 15 & 1 2 7 & 1 2 24 & 1 3 24 & 1 3 5 \\ \hline
0 & 3 & 4 & 10 & 1 & 26 & 17 & 23 & 24 & 1 2 25 & 1 2 8 & 1 4 10 & 1 3 12 & 1 3 19 & 1 4 22 & 1 2 22 & 1 2 5 & 1 4 7 \\ \hline
0 & 4 & 9 & 1 & 8 & 19 & 26 & 18 & 23 & 1 5 10 & 1 4 9 & 1 2 21 & 1 8 19 & 1 8 17 & 1 2 9 & 1 4 23 & 1 5 23 & 1 5 9 \\ \hline
0 & 5 & 7 & 13 & 2 & 25 & 14 & 20 & 22 & 1 3 23 & 1 3 9 & 1 6 12 & 1 5 16 & 1 5 17 & 1 6 22 & 1 3 22 & 1 3 8 & 1 6 11 \\ \hline
0 & 6 & 10 & 18 & 1 & 26 & 9 & 17 & 21 & 1 5 22 & 1 5 13 & 1 9 19 & 1 3 20 & 1 3 11 & 1 9 18 & 1 5 20 & 1 5 11 & 1 7 13 \\ \hline
0 & 7 & 9 & 19 & 5 & 22 & 8 & 18 & 20 & 1 3 21 & 1 3 13 & 1 5 15 & 1 4 14 & 1 4 18 & 1 5 18 & 1 3 18 & 1 3 10 & 1 8 15 \\ \hline
0 & 8 & 9 & 18 & 4 & 23 & 9 & 18 & 19 & 1 2 20 & 1 2 11 & 1 6 15 & 1 6 14 & 1 6 20 & 1 6 19 & 1 2 19 & 1 2 10 & 1 9 17 \\ \hline
0 & 10 & 3 & 15 & 1 & 26 & 12 & 24 & 17 & 1 4 11 & 1 6 21 & 1 3 15 & 1 3 17 & 1 3 14 & 1 3 16 & 1 6 13 & 1 4 21 & 1 8 18 \\ \hline
0 & 11 & 1 & 22 & 8 & 19 & 5 & 26 & 16 & 1 2 12 & 1 7 17 & 1 7 14 & 1 4 17 & 1 4 15 & 1 7 21 & 1 7 18 & 1 2 18 & 1 6 17 \\ \hline
0 & 12 & 9 & 1 & 16 & 11 & 26 & 18 & 15 & 1 4 19 & 1 4 20 & 1 8 20 & 1 6 18 & 1 6 16 & 1 8 16 & 1 4 12 & 1 4 13 & 1 4 16 \\ \hline
0 & 13 & 4 & 19 & 20 & 7 & 8 & 23 & 14 & 1 5 14 & 1 7 19 & 1 2 13 & 1 2 16 & 1 2 14 & 1 2 17 & 1 7 16 & 1 5 19 & 1 2 15 \\ \hline
\end{tabular}
\label{t:C9:27}
\end{table}

\begin{table}[htbp]
\centering
\small
\caption{Base cycles of system $\CCC{9}{30}$ }
\begin{tabular}{|c|c|c|c|c|c|c|c|c||c|c|c|c|c|c|c|c|c|}
\hline
\multicolumn{ 9}{|c||}{Vertices} & \multicolumn{ 9}{c|}{Triplet types} \\ \hline
\textbf{1} & \textbf{2} & \textbf{3} & \textbf{4} & \textbf{5} & \textbf{6} & \textbf{7} & \textbf{8} & \textbf{9} & \textbf{1-2-3} & \textbf{2-3-4} & \textbf{3-4-5} & \textbf{4-5-6} & \textbf{5-6-7} & \textbf{6-7-8} & \textbf{7-8-9} & \textbf{8-9-1} & \textbf{9-1-2} \\ \hline
\hline
0 & 1 & 3 & 6 & 7 & 23 & 24 & 27 & 29 & 1 2 4 & 1 3 6 & 1 2 28 & 1 2 18 & 1 2 15 & 1 2 5 & 1 3 28 & 1 2 29 & 1 2 3 \\ \hline
0 & 2 & 6 & 1 & 7 & 23 & 29 & 24 & 28 & 1 3 7 & 1 2 6 & 1 2 26 & 1 7 23 & 1 7 15 & 1 2 7 & 1 2 27 & 1 3 27 & 1 3 5 \\ \hline
0 & 3 & 7 & 12 & 1 & 29 & 18 & 23 & 27 & 1 4 8 & 1 5 10 & 1 6 25 & 1 3 14 & 1 3 20 & 1 6 12 & 1 5 26 & 1 4 27 & 1 4 7 \\ \hline
0 & 4 & 10 & 1 & 8 & 22 & 29 & 20 & 26 & 1 5 11 & 1 4 10 & 1 3 24 & 1 8 22 & 1 8 17 & 1 3 10 & 1 4 25 & 1 5 25 & 1 5 9 \\ \hline
0 & 5 & 7 & 13 & 6 & 24 & 17 & 23 & 25 & 1 3 26 & 1 3 9 & 1 2 8 & 1 8 19 & 1 8 20 & 1 2 25 & 1 3 25 & 1 3 8 & 1 6 11 \\ \hline
0 & 6 & 13 & 1 & 4 & 26 & 29 & 17 & 24 & 1 7 14 & 1 6 13 & 1 4 13 & 1 4 26 & 1 4 9 & 1 4 22 & 1 6 24 & 1 7 24 & 1 7 13 \\
\hline
\end{tabular}
\label{t:C9:30}
\end{table}




\newpage
\section*{Appendix C: Detailed tables --- \hfill\break
\hspace*{5.7em} 6-cycle decompositions}

\begin{table}[htbp]
\caption{Base cycles and triplet type calculations of system $\CCC{6}{12}$ }

\begin{subtable}{1\textwidth}
\centering
\caption{The first three triplets of the base cycles, denoted by 1-2-3, 2-3-4 and 3-4-5.}
\begin{tabular}{|c||f|c|c|f|c|c|f|c|c|}
\hline
\multirow{2}{*}{base cycle} & \multicolumn{3}{|c|}{(1-2-3)}& \multicolumn{3}{|c}{(2-3-4)} & \multicolumn{3}{|c|}{(3-4-5)} \\ \cline{2-10}
& d & triplet & type & d & triplet & type & d & triplet & type \\ \hline \hline
(0,1,2,4,5,8) & 1 & 0 1 2 & 1 2 3 & 1 & 1 2 4 & 1 2 4 & 1 & 4 5 2 & 1 2 11 \\ \hline
(0,1,5,8,3,6) & 1 & 0 1 5 & 1 2 6 & 3 & 5 8 1 & 1 4 9 & 2 & 3 5 8 & 1 3 6 \\ \hline
(0,1,9,3,5,7) & 1 & 0 1 9 & 1 2 10 & 2 & 1 3 9 & 1 3 9 & 2 & 3 5 9 & 1 3 7 \\ \hline
 \end{tabular}
\label{6--12--a}
\end{subtable}

\vspace*{0.5 cm}

\begin{subtable}{1\textwidth}
\centering
\caption{The last three triplets of the base cycles, denoted by 4-5-6, 5-6-1 and 6-1-2.}
\begin{tabular}{|c||f|c|c|f|c|c|f|c|c|}
\hline
\multirow{2}{*}{base cycle}  & \multicolumn{3}{|c}{(4-5-6)}& \multicolumn{3}{|c}{(5-6-1)} & \multicolumn{3}{|c|}{(6-1-2)} \\ \cline{2-10}
& d & triplet & type & d & triplet & type & d & triplet & type \\ \hline \hline
(0,1,2,4,5,8) & 1 & 4 5 8 & 1 2 5 & 3 & 5 8 0 & 1 4 8 & 1 & 0 1 8 & 1 2 9 \\ \hline
(0,1,5,8,3,6) & 2 & 6 8 3 & 1 3 10 & 3 & 0 3 6 & 1 4 7 & 1 & 0 1 6 & 1 2 7 \\ \hline
(0,1,9,3,5,7) & 2 & 3 5 7 & 1 3 5 & 2 & 5 7 0 & 1 3 8 & 1 & 0 1 7 & 1 2 8 \\ \hline
\end{tabular}
\label{6--12--b}
\end{subtable}
\end{table}


\begin{table}[htbp]
\caption{Base cycles and triplet type calculations of system $\CCC{6}{15}$ }

\begin{subtable}{1\textwidth}
\centering
\caption{The first three triplets of the base cycles, denoted by 1-2-3, 2-3-4 and 3-4-5.}
\begin{tabular}{|c||f|c|c|f|c|c|f|c|c|}
\hline
\multirow{2}{*}{base cycle} & \multicolumn{3}{|c|}{(1-2-3)}& \multicolumn{3}{|c}{(2-3-4)} & \multicolumn{3}{|c|}{(3-4-5)} \\ \cline{2-10}
& d & triplet & type & d & triplet & type & d & triplet & type \\ \hline \hline
(0,1,2,4,5,8) & 1 & 0 1 2 & 1 2 3 & 1 & 1 2 4 & 1 2 4 & 1 & 4 5 2 & 1 2 14 \\ \hline
(0,1,5,3,8,6) & 1 & 0 1 5 & 1 2 6 & 2 & 1 3 5 & 1 3 5 & 2 & 3 5 8 & 1 3 6 \\ \hline
(0,1,7,3,6,10) & 1 & 0 1 7 & 1 2 8 & 2 & 1 3 7 & 1 3 7 & 1 & 6 7 3 & 1 2 13 \\ \hline
(0,1,9,14,7,11) & 1 & 0 1 9 & 1 2 10 & 2 & 14 1 9 & 1 3 11 & 2 & 7 9 14 & 1 3 8 \\ \hline
(0,3,12,6,2,8) & 3 & 12 0 3 & 1 4 7 & 3 & 3 6 12 & 1 4 10 & 4 & 2 6 12 & 1 5 11 \\ \hline
 \end{tabular}
\label{6--15--a}
\end{subtable}

\vspace*{0.5 cm}

\begin{subtable}{1\textwidth}
\centering
\caption{The last three triplets of the base cycles, denoted by 4-5-6, 5-6-1 and 6-1-2.}
\begin{tabular}{|c||f|c|c|f|c|c|f|c|c|}
\hline
\multirow{2}{*}{base cycle}  & \multicolumn{3}{|c}{(4-5-6)}& \multicolumn{3}{|c}{(5-6-1)} & \multicolumn{3}{|c|}{(6-1-2)} \\ \cline{2-10}
& d & triplet & type & d & triplet & type & d & triplet & type \\ \hline \hline
(0,1,2,4,5,8) & 1 & 4 5 8 & 1 2 5 & 3 & 5 8 0 & 1 4 11 & 1 & 0 1 8 & 1 2 9 \\ \hline
(0,1,5,3,8,6) & 2 & 6 8 3 & 1 3 13 & 2 & 6 8 0 & 1 3 10 & 1 & 0 1 6 & 1 2 7 \\ \hline
(0,1,7,3,6,10) & 3 & 3 6 10 & 1 4 8 & 4 & 6 10 0 & 1 5 10 & 1 & 0 1 10 & 1 2 11 \\ \hline
(0,1,9,14,7,11) & 3 & 11 14 7 & 1 4 12 & 4 & 7 11 0 & 1 5 9 & 1 & 0 1 11 & 1 2 12 \\ \hline
(0,3,12,6,2,8) & 2 & 6 8 2 & 1 3 12 & 2 & 0 2 8 & 1 3 9 & 3 & 0 3 8 & 1 4 9 \\ \hline
\end{tabular}
\label{6--15--b}
\end{subtable}
\end{table}


\begin{table}[htbp]
\caption{Base cycles and triplet type calculations of system $\CCC{6}{18}$ }

\begin{subtable}{1\textwidth}
\centering
\caption{The first three triplets of the base cycles, denoted by 1-2-3, 2-3-4 and 3-4-5.}
\begin{tabular}{|c||f|c|c|f|c|c|f|c|c|}
\hline
\multirow{2}{*}{base cycle} & \multicolumn{3}{|c|}{(1-2-3)}& \multicolumn{3}{|c}{(2-3-4)} & \multicolumn{3}{|c|}{(3-4-5)} \\ \cline{2-10}
& d & triplet & type & d & triplet & type & d & triplet & type \\ \hline \hline
\multicolumn{10}{|l|}{Base cycle generating 9 cycles:} \\ \hline
(0,1,2,9,10,11) & 1 & 0 1 2 & 1 2 3 & 1 & 1 2 9 & 1 2 9 & 1 & 9 10 2 & 1 2 12 \\ \hline
\multicolumn{10}{|l|}{Base cycles generating 18 cycles:} \\ \hline
(0,1,3,4,7,9) & 1 & 0 1 3 & 1 2 4 & 1 & 3 4 1 & 1 2 17 & 1 & 3 4 7 & 1 2 5 \\ \hline
(0,1,5,3,8,10) & 1 & 0 1 5 & 1 2 6 & 2 & 1 3 5 & 1 3 5 & 2 & 3 5 8 & 1 3 6 \\ \hline
(0,1,6,5,2,12) & 1 & 0 1 6 & 1 2 7 & 1 & 5 6 1 & 1 2 15 & 1 & 5 6 2 & 1 2 16 \\ \hline
(0,1,7,3,9,13) & 1 & 0 1 7 & 1 2 8 & 2 & 1 3 7 & 1 3 7 & 2 & 7 9 3 & 1 3 15 \\ \hline
(0,2,7,13,3,8) & 2 & 0 2 7 & 1 3 8 & 5 & 2 7 13 & 1 6 12 & 4 & 3 7 13 & 1 5 11 \\ \hline
(0,3,7,17,4,11) & 3 & 0 3 7 & 1 4 8 & 4 & 17 3 7 & 1 5 9 & 3 & 4 7 17 & 1 4 14 \\ \hline
(0,3,12,8,15,6) & 3 & 0 3 12 & 1 4 13 & 4 & 8 12 3 & 1 5 14 & 3 & 12 15 8 & 1 4 15 \\ \hline
 \end{tabular}
\label{6--18--a}
\end{subtable}

\vspace*{0.5 cm}

\begin{subtable}{1\textwidth}
\centering
\caption{The last three triplets of the base cycles, denoted by 4-5-6, 5-6-1 and 6-1-2.}
\begin{tabular}{|c||f|c|c|f|c|c|f|c|c|}
\hline
\multirow{2}{*}{base cycle}  & \multicolumn{3}{|c}{(4-5-6)}& \multicolumn{3}{|c}{(5-6-1)} & \multicolumn{3}{|c|}{(6-1-2)} \\ \cline{2-10}
& d & triplet & type & d & triplet & type & d & triplet & type \\ \hline \hline
\multicolumn{10}{|l|}{Base cycle generating 9 cycles:} \\ \hline
(0,1,2,9,10,11) & 1 & 9 10 11 & 1 2 3 & 1 & 10 11 0 & 1 2 9 & 1 & 0 1 11 & 1 2 12 \\ \hline
\multicolumn{10}{|l|}{Base cycles generating 18 cycles:} \\ \hline
(0,1,3,4,7,9) & 2 & 7 9 4 & 1 3 16 & 2 & 7 9 0 & 1 3 12 & 1 & 0 1 9 & 1 2 10 \\ \hline
(0,1,5,3,8,10) & 2 & 8 10 3 & 1 3 14 & 2 & 8 10 0 & 1 3 11 & 1 & 0 1 10 & 1 2 11 \\ \hline
(0,1,6,5,2,12) & 3 & 2 5 12 & 1 4 11 & 2 & 0 2 12 & 1 3 13 & 1 & 0 1 12 & 1 2 13 \\ \hline
(0,1,7,3,9,13) & 4 & 9 13 3 & 1 5 13 & 4 & 9 13 0 & 1 5 10 & 1 & 0 1 13 & 1 2 14 \\ \hline
(0,2,7,13,3,8) & 5 & 3 8 13 & 1 6 11 & 3 & 0 3 8 & 1 4 9 & 2 & 0 2 8 & 1 3 9 \\ \hline
(0,3,7,17,4,11) & 5 & 17 4 11 & 1 6 13 & 4 & 0 4 11 & 1 5 12 & 3 & 0 3 11 & 1 4 12 \\ \hline
(0,3,12,8,15,6) & 2 & 6 8 15 & 1 3 10 & 3 & 15 0 6 & 1 4 10 & 3 & 0 3 6 & 1 4 7 \\ \hline
\end{tabular}
\label{6--18--b}
\end{subtable}
\end{table}


\begin{table}[htbp]
\caption{Base cycles and triplet type calculations of system $\CCC{6}{24}$ }

\begin{subtable}{1\textwidth}
\centering
\caption{The first three triplets of the base cycles, denoted by 1-2-3, 2-3-4 and 3-4-5.}
\begin{tabular}{|c||f|c|c|f|c|c|f|c|c|}
\hline
\multirow{2}{*}{base cycle} & \multicolumn{3}{|c|}{(1-2-3)}& \multicolumn{3}{|c}{(2-3-4)} & \multicolumn{3}{|c|}{(3-4-5)} \\ \cline{2-10}
& d & triplet & type & d & triplet & type & d & triplet & type \\ \hline \hline
(0,1,2,4,5,8) & 1 & 0 1 2 & 1 2 3 & 1 & 1 2 4 & 1 2 4 & 1 & 4 5 2 & 1 2 23 \\ \hline
(0,1,5,3,8,6) & 1 & 0 1 5 & 1 2 6 & 2 & 1 3 5 & 1 3 5 & 2 & 3 5 8 & 1 3 6 \\ \hline
(0,1,7,3,6,9) & 1 & 0 1 7 & 1 2 8 & 2 & 1 3 7 & 1 3 7 & 1 & 6 7 3 & 1 2 22 \\ \hline
(0,1,10,3,5,11) & 1 & 0 1 10 & 1 2 11 & 2 & 1 3 10 & 1 3 10 & 2 & 3 5 10 & 1 3 8 \\ \hline
(0,1,12,3,6,13) & 1 & 0 1 12 & 1 2 13 & 2 & 1 3 12 & 1 3 12 & 3 & 3 6 12 & 1 4 10 \\ \hline
(0,1,14,3,6,15) & 1 & 0 1 14 & 1 2 15 & 2 & 1 3 14 & 1 3 14 & 3 & 3 6 14 & 1 4 12 \\ \hline
(0,1,16,3,6,17) & 1 & 0 1 16 & 1 2 17 & 2 & 1 3 16 & 1 3 16 & 3 & 3 6 16 & 1 4 14 \\ \hline
(0,1,18,3,6,20) & 1 & 0 1 18 & 1 2 19 & 2 & 1 3 18 & 1 3 18 & 3 & 3 6 18 & 1 4 16 \\ \hline
(0,2,10,5,1,19) & 2 & 0 2 10 & 1 3 11 & 3 & 2 5 10 & 1 4 9 & 4 & 1 5 10 & 1 5 10 \\ \hline
(0,2,12,7,16,20) & 2 & 0 2 12 & 1 3 13 & 5 & 2 7 12 & 1 6 11 & 4 & 12 16 7 & 1 5 20 \\ \hline
(0,2,14,7,3,16) & 2 & 0 2 14 & 1 3 15 & 5 & 2 7 14 & 1 6 13 & 4 & 3 7 14 & 1 5 12 \\ \hline
(0,4,12,17,3,7) & 4 & 0 4 12 & 1 5 13 & 5 & 12 17 4 & 1 6 17 & 5 & 12 17 3 & 1 6 16 \\ \hline
(0,4,17,23,7,16) & 4 & 0 4 17 & 1 5 18 & 5 & 23 4 17 & 1 6 19 & 6 & 17 23 7 & 1 7 15 \\ \hline
(0,5,13,19,7,14) & 5 & 0 5 13 & 1 6 14 & 6 & 13 19 5 & 1 7 17 & 6 & 7 13 19 & 1 7 13 \\ \hline
 \end{tabular}
\label{6--24--a}
\end{subtable}

\vspace*{0.5 cm}

\begin{subtable}{1\textwidth}
\centering
\caption{The last three triplets of the base cycles, denoted by 4-5-6, 5-6-1 and 6-1-2.}
\begin{tabular}{|c||f|c|c|f|c|c|f|c|c|}
\hline
\multirow{2}{*}{base cycle}  & \multicolumn{3}{|c}{(4-5-6)}& \multicolumn{3}{|c}{(5-6-1)} & \multicolumn{3}{|c|}{(6-1-2)} \\ \cline{2-10}
& d & triplet & type & d & triplet & type & d & triplet & type \\ \hline \hline
(0,1,2,4,5,8) & 1 & 4 5 8 & 1 2 5 & 3 & 5 8 0 & 1 4 20 & 1 & 0 1 8 & 1 2 9 \\ \hline
(0,1,5,3,8,6) & 2 & 6 8 3 & 1 3 22 & 2 & 6 8 0 & 1 3 19 & 1 & 0 1 6 & 1 2 7 \\ \hline
(0,1,7,3,6,9) & 3 & 3 6 9 & 1 4 7 & 3 & 6 9 0 & 1 4 19 & 1 & 0 1 9 & 1 2 10 \\ \hline
(0,1,10,3,5,11) & 2 & 3 5 11 & 1 3 9 & 5 & 0 5 11 & 1 6 12 & 1 & 0 1 11 & 1 2 12 \\ \hline
(0,1,12,3,6,13) & 3 & 3 6 13 & 1 4 11 & 6 & 0 6 13 & 1 7 14 & 1 & 0 1 13 & 1 2 14 \\ \hline
(0,1,14,3,6,15) & 3 & 3 6 15 & 1 4 13 & 6 & 0 6 15 & 1 7 16 & 1 & 0 1 15 & 1 2 16 \\ \hline
(0,1,16,3,6,17) & 3 & 3 6 17 & 1 4 15 & 6 & 0 6 17 & 1 7 18 & 1 & 0 1 17 & 1 2 18 \\ \hline
(0,1,18,3,6,20) & 3 & 3 6 20 & 1 4 18 & 4 & 20 0 6 & 1 5 11 & 1 & 0 1 20 & 1 2 21 \\ \hline
(0,2,10,5,1,19) & 4 & 1 5 19 & 1 5 19 & 1 & 0 1 19 & 1 2 20 & 2 & 0 2 19 & 1 3 20 \\ \hline
(0,2,12,7,16,20) & 4 & 16 20 7 & 1 5 16 & 4 & 16 20 0 & 1 5 9 & 2 & 0 2 20 & 1 3 21 \\ \hline
(0,2,14,7,3,16) & 4 & 3 7 16 & 1 5 14 & 3 & 0 3 16 & 1 4 17 & 2 & 0 2 16 & 1 3 17 \\ \hline
(0,4,12,17,3,7) & 4 & 3 7 17 & 1 5 15 & 3 & 0 3 7 & 1 4 8 & 3 & 4 7 0 & 1 4 21 \\ \hline
(0,4,17,23,7,16) & 7 & 16 23 7 & 1 8 16 & 7 & 0 7 16 & 1 8 17 & 4 & 0 4 16 & 1 5 17 \\ \hline
(0,5,13,19,7,14) & 5 & 14 19 7 & 1 6 18 & 7 & 0 7 14 & 1 8 15 & 5 & 0 5 14 & 1 6 15 \\ \hline
\end{tabular}
\label{6--24--b}
\end{subtable}
\end{table}


\begin{table}[htbp]
\caption{Base cycles and triplet type calculations of system $\CCC{6}{27}$ }

\begin{subtable}{1\textwidth}
\centering
\caption{The first three triplets of the base cycles, denoted by 1-2-3, 2-3-4 and 3-4-5.}
\begin{tabular}{|c||f|c|c|f|c|c|f|c|c|}
\hline
\multirow{2}{*}{base cycle} & \multicolumn{3}{|c|}{(1-2-3)}& \multicolumn{3}{|c}{(2-3-4)} & \multicolumn{3}{|c|}{(3-4-5)} \\ \cline{2-10}
& d & triplet & type & d & triplet & type & d & triplet & type \\ \hline \hline
(0,1,3,5,24,26) & 1 & 0 1 3 & 1 2 4 & 2 & 1 3 5 & 1 3 5 & 2 & 3 5 24 & 1 3 22 \\ \hline
(0,3,4,8,23,24) & 1 & 3 4 0 & 1 2 25 & 1 & 3 4 8 & 1 2 6 & 4 & 4 8 23 & 1 5 20 \\ \hline
(0,4,6,1,21,23) & 2 & 4 6 0 & 1 3 24 & 2 & 4 6 1 & 1 3 25 & 5 & 1 6 21 & 1 6 21 \\ \hline
(0,5,6,12,21,22) & 1 & 5 6 0 & 1 2 23 & 1 & 5 6 12 & 1 2 8 & 6 & 6 12 21 & 1 7 16 \\ \hline
(0,6,9,1,18,21) & 3 & 6 9 0 & 1 4 22 & 3 & 6 9 1 & 1 4 23 & 8 & 1 9 18 & 1 9 18 \\ \hline
(0,7,3,8,24,20) & 3 & 0 3 7 & 1 4 8 & 1 & 7 8 3 & 1 2 24 & 5 & 3 8 24 & 1 6 22 \\ \hline
(0,8,1,10,26,19) & 1 & 0 1 8 & 1 2 9 & 2 & 8 10 1 & 1 3 21 & 2 & 26 1 10 & 1 3 12 \\ \hline
(0,10,2,11,25,17) & 2 & 0 2 10 & 1 3 11 & 1 & 10 11 2 & 1 2 20 & 4 & 25 2 11 & 1 5 14 \\ \hline
(0,11,1,8,26,16) & 1 & 0 1 11 & 1 2 12 & 3 & 8 11 1 & 1 4 21 & 2 & 26 1 8 & 1 3 10 \\ \hline
(0,12,2,16,25,15) & 2 & 0 2 12 & 1 3 13 & 4 & 12 16 2 & 1 5 18 & 4 & 25 2 16 & 1 5 19 \\ \hline
(0,13,1,4,26,14) & 1 & 0 1 13 & 1 2 14 & 3 & 1 4 13 & 1 4 13 & 2 & 26 1 4 & 1 3 6 \\ \hline
(0,1,9,4,13,16) & 1 & 0 1 9 & 1 2 10 & 3 & 1 4 9 & 1 4 9 & 4 & 9 13 4 & 1 5 23 \\ \hline
(0,2,13,5,18,22) & 2 & 0 2 13 & 1 3 14 & 3 & 2 5 13 & 1 4 12 & 5 & 13 18 5 & 1 6 20 \\ \hline
(0,2,16,23,5,18) & 2 & 0 2 16 & 1 3 17 & 6 & 23 2 16 & 1 7 21 & 7 & 16 23 5 & 1 8 17 \\ \hline
(0,3,16,6,21,17) & 3 & 0 3 16 & 1 4 17 & 3 & 3 6 16 & 1 4 14 & 5 & 16 21 6 & 1 6 18 \\ \hline
(0,4,12,19,7,21) & 4 & 0 4 12 & 1 5 13 & 7 & 12 19 4 & 1 8 20 & 5 & 7 12 19 & 1 6 13 \\ \hline
(0,6,17,23,12,19) & 6 & 0 6 17 & 1 7 18 & 6 & 17 23 6 & 1 7 17 & 5 & 12 17 23 & 1 6 12 \\ \hline
(0,6,18,4,23,7) & 6 & 0 6 18 & 1 7 19 & 2 & 4 6 18 & 1 3 15 & 5 & 18 23 4 & 1 6 14 \\ \hline
 \end{tabular}
\label{6--27--a}
\end{subtable}
\end{table}


\begin{table}
\ContinuedFloat
\begin{subtable}{1\textwidth}
\centering
\caption{The last three triplets of the base cycles, denoted by 4-5-6, 5-6-1 and 6-1-2.}
\begin{tabular}{|c||f|c|c|f|c|c|f|c|c|}
\hline
\multirow{2}{*}{base cycle}  & \multicolumn{3}{|c}{(4-5-6)}& \multicolumn{3}{|c}{(5-6-1)} & \multicolumn{3}{|c|}{(6-1-2)} \\ \cline{2-10}
& d & triplet & type & d & triplet & type & d & triplet & type \\ \hline \hline
(0,1,3,5,24,26) & 2 & 24 26 5 & 1 3 9 & 1 & 26 0 24 & 1 2 26 & 1 & 26 0 1 & 1 2 3 \\ \hline
(0,3,4,8,23,24) & 1 & 23 24 8 & 1 2 13 & 1 & 23 24 0 & 1 2 5 & 3 & 24 0 3 & 1 4 7 \\ \hline
(0,4,6,1,21,23) & 2 & 21 23 1 & 1 3 8 & 2 & 21 23 0 & 1 3 7 & 4 & 23 0 4 & 1 5 9 \\ \hline
(0,5,6,12,21,22) & 1 & 21 22 12 & 1 2 19 & 1 & 21 22 0 & 1 2 7 & 5 & 22 0 5 & 1 6 11 \\ \hline
(0,6,9,1,18,21) & 3 & 18 21 1 & 1 4 11 & 3 & 18 21 0 & 1 4 10 & 6 & 21 0 6 & 1 7 13 \\ \hline
(0,7,3,8,24,20) & 4 & 20 24 8 & 1 5 16 & 3 & 24 0 20 & 1 4 24 & 7 & 20 0 7 & 1 8 15 \\ \hline
(0,8,1,10,26,19) & 7 & 19 26 10 & 1 8 19 & 1 & 26 0 19 & 1 2 21 & 8 & 19 0 8 & 1 9 17 \\ \hline
(0,10,2,11,25,17) & 6 & 11 17 25 & 1 7 15 & 2 & 25 0 17 & 1 3 20 & 7 & 10 17 0 & 1 8 18 \\ \hline
(0,11,1,8,26,16) & 8 & 8 16 26 & 1 9 19 & 1 & 26 0 16 & 1 2 18 & 5 & 11 16 0 & 1 6 17 \\ \hline
(0,12,2,16,25,15) & 1 & 15 16 25 & 1 2 11 & 2 & 25 0 15 & 1 3 18 & 3 & 12 15 0 & 1 4 16 \\ \hline
(0,13,1,4,26,14) & 5 & 26 4 14 & 1 6 16 & 1 & 26 0 14 & 1 2 16 & 1 & 13 14 0 & 1 2 15 \\ \hline
(0,1,9,4,13,16) & 3 & 13 16 4 & 1 4 19 & 3 & 13 16 0 & 1 4 15 & 1 & 0 1 16 & 1 2 17 \\ \hline
(0,2,13,5,18,22) & 4 & 18 22 5 & 1 5 15 & 4 & 18 22 0 & 1 5 10 & 2 & 0 2 22 & 1 3 23 \\ \hline
(0,2,16,23,5,18) & 5 & 18 23 5 & 1 6 15 & 5 & 0 5 18 & 1 6 19 & 2 & 0 2 18 & 1 3 19 \\ \hline
(0,3,16,6,21,17) & 4 & 17 21 6 & 1 5 17 & 4 & 17 21 0 & 1 5 11 & 3 & 0 3 17 & 1 4 18 \\ \hline
(0,4,12,19,7,21) & 2 & 19 21 7 & 1 3 16 & 6 & 21 0 7 & 1 7 14 & 4 & 0 4 21 & 1 5 22 \\ \hline
(0,6,17,23,12,19) & 4 & 19 23 12 & 1 5 21 & 7 & 12 19 0 & 1 8 16 & 6 & 0 6 19 & 1 7 20 \\ \hline
(0,6,18,4,23,7) & 3 & 4 7 23 & 1 4 20 & 4 & 23 0 7 & 1 5 12 & 1 & 6 7 0 & 1 2 22 \\ \hline
\end{tabular}
\label{6--27--b}
\end{subtable}
\end{table}


\begin{table}[htbp]
\caption{Base cycles and triplet type calculations of system $\CCC{6}{30}$ }

\begin{subtable}{1\textwidth}
\centering
\caption{The first three triplets of the base cycles, denoted by 1-2-3, 2-3-4 and 3-4-5.}
\begin{tabular}{|c||f|c|c|f|c|c|f|c|c|}
\hline
\multirow{2}{*}{base cycle} & \multicolumn{3}{|c|}{(1-2-3)}& \multicolumn{3}{|c}{(2-3-4)} & \multicolumn{3}{|c|}{(3-4-5)} \\ \cline{2-10}
& d & triplet & type & d & triplet & type & d & triplet & type \\ \hline \hline
\multicolumn{10}{|l|}{Base cycle generating 15 cycles:} \\ \hline
(0,1,2,15,16,17) & 1 & 0 1 2 & 1 2 3 & 1 & 1 2 15 & 1 2 15 & 1 & 15 16 2 & 1 2 18 \\ \hline
\multicolumn{10}{|l|}{Base cycles generating 30 cycles:} \\ \hline
(0,2,3,15,27,28) & 1 & 2 3 0 & 1 2 29 & 1 & 2 3 15 & 1 2 14 & 6 & 27 3 15 & 1 7 19 \\ \hline
(0,3,4,15,26,27) & 1 & 3 4 0 & 1 2 28 & 1 & 3 4 15 & 1 2 13 & 8 & 26 4 15 & 1 9 20 \\ \hline
(0,4,6,15,24,26) & 2 & 4 6 0 & 1 3 27 & 2 & 4 6 15 & 1 3 12 & 9 & 6 15 24 & 1 10 19 \\
\hline
 \end{tabular}
\label{6--30--a}
\end{subtable}
\end{table}

\begin{table}
\ContinuedFloat
\begin{subtable}{1\textwidth}
\centering
\caption{The last three triplets of the base cycles, denoted by 4-5-6, 5-6-1 and 6-1-2.}
\begin{tabular}{|c||f|c|c|f|c|c|f|c|c|}
\hline
\multirow{2}{*}{base cycle}  & \multicolumn{3}{|c}{(4-5-6)}& \multicolumn{3}{|c}{(5-6-1)} & \multicolumn{3}{|c|}{(6-1-2)} \\ \cline{2-10}
& d & triplet & type & d & triplet & type & d & triplet & type \\ \hline \hline
\multicolumn{10}{|l|}{Base cycle generating 15 cycles:} \\ \hline
(0,1,2,15,16,17) & 1 & 15 16 17 & 1 2 3 & 1 & 16 17 0 & 1 2 15 & 1 & 0 1 17 & 1 2 18 \\ \hline
\multicolumn{10}{|l|}{Base cycles generating 30 cycles:} \\ \hline
(0,2,3,15,27,28) & 1 & 27 28 15 & 1 2 19 & 1 & 27 28 0 & 1 2 4 & 2 & 28 0 2 & 1 3 5 \\ \hline
(0,3,4,15,26,27) & 1 & 26 27 15 & 1 2 20 & 1 & 26 27 0 & 1 2 5 & 3 & 27 0 3 & 1 4 7 \\ \hline
(0,4,6,15,24,26) & 2 & 24 26 15 & 1 3 22 & 2 & 24 26 0 & 1 3 7 & 4 & 26 0 4 & 1 5 9 \\
\hline
\end{tabular}
\label{6--30--b}
\end{subtable}
\end{table}

\newpage
\section*{Appendix D: Detailed tables --- \hfill\break
\hspace*{5.7em} 9-cycle decompositions}

\begin{table}[htbp]
\caption{Base cycle and triplet type calculations of system $\CCC{9}{9}$ }

\begin{subtable}{1\textwidth}
\centering
\caption{The first three triplets of the base cycles, denoted by 1-2-3, 2-3-4 and 3-4-5.}
\begin{tabular}{|c||f|c|c|f|c|c|f|c|c|}
\hline
\multirow{2}{*}{base cycle} & \multicolumn{3}{|c}{(1-2-3)}& \multicolumn{3}{|c}{(2-3-4)} & \multicolumn{3}{|c|}{(3-4-5)} \\ \cline{2-10}
& d & triplet & type & d & triplet & type & d & triplet & type \\ \hline \hline
(0,1,2,7,6,4,8,5,3) & 1 & 0 1 2 & 1 2 3 & 1 & 1 2 7 & 1 2 7 & 1 & 6 7 2 & 1 2 6 \\ \hline
\end{tabular}
\label{9--9--a}
\end{subtable}

\vspace{0.2 cm}

\begin{subtable}{1\textwidth}
\centering
\caption{The next three triplets of the base cycles, denoted by 4-5-6, 5-6-7 and 6-7-8.}
\begin{tabular}{|c||f|c|c|f|c|c|f|c|c|}
\hline
\multirow{2}{*}{base cycle}& \multicolumn{3}{|c}{(4-5-6)}& \multicolumn{3}{|c}{(5-6-7)} & \multicolumn{3}{|c|}{(6-7-8)} \\ \cline{2-10}
& d & triplet & type & d & triplet & type & d & triplet & type \\ \hline \hline
(0,1,2,7,6,4,8,5,3) & 1 & 6 7 4 & 1 2 8 & 2 & 4 6 8 & 1 3 5 & 1 & 4 5 8 & 1 2 5 \\ \hline
\end{tabular}
\label{9--9--b}
\end{subtable}

\vspace{0.2 cm}

\ContinuedFloat
\begin{subtable}{1\textwidth}
\centering
\caption{The last three triplets of the base cycles, denoted by 7-8-9, 8-9-1 and 9-1-2.}
\begin{tabular}{|c||f|c|c|f|c|c|f|c|c|}
\hline
\multirow{2}{*}{base cycle} & \multicolumn{3}{|c}{(7-8-9)}& \multicolumn{3}{|c}{(8-9-1)} & \multicolumn{3}{|c|}{(9-1-2)} \\ \cline{2-10}
& d & triplet & type & d & triplet & type & d & triplet & type \\ \hline \hline
(0,1,2,7,6,4,8,5,3) & 2 & 3 5 8 & 1 3 6 & 2 & 3 5 0 & 1 3 7 & 1 & 0 1 3 & 1 2 4 \\ \hline
\end{tabular}
\label{9--9--c}
\end{subtable}
\end{table}


\begin{table}[htbp]
\caption{Base cycles and triplet type calculations of system $\CCC{9}{12}$ }

\begin{subtable}{1\textwidth}
\centering
\caption{The first three triplets of the base cycles, denoted by 1-2-3, 2-3-4 and 3-4-5.}
\begin{tabular}{|c||f|c|c|f|c|c|f|c|c|}
\hline
\multirow{2}{*}{base cycle} & \multicolumn{3}{|c}{(1-2-3)}& \multicolumn{3}{|c}{(2-3-4)} & \multicolumn{3}{|c|}{(3-4-5)} \\ \cline{2-10}
& d & triplet & type & d & triplet & type & d & triplet & type \\ \hline \hline
(0,1,2,4,5,8,10,3,6) & 1 & 0 1 2 & 1 2 3 & 1 & 1 2 4 & 1 2 4 & 1 & 4 5 2 & 1 2 11 \\ \hline
(0,1,7,11,2,9,10,6,8) & 1 & 0 1 7 & 1 2 8 & 2 & 11 1 7 & 1 3 9 & 3 & 11 2 7 & 1 4 9 \\ \hline
\end{tabular}
\label{9--12--a}
\end{subtable}

\vspace{0.2 cm}

\begin{subtable}{1\textwidth}
\centering
\caption{The next three triplets of the base cycles, denoted by 4-5-6, 5-6-7 and 6-7-8.}
\begin{tabular}{|c||f|c|c|f|c|c|f|c|c|}
\hline
\multirow{2}{*}{base cycle} & \multicolumn{3}{|c}{(4-5-6)}& \multicolumn{3}{|c}{(5-6-7)} & \multicolumn{3}{|c|}{(6-7-8)} \\ \cline{2-10}
& d & triplet & type & d & triplet & type & d & triplet & type \\ \hline \hline
(0,1,2,4,5,8,10,3,6) & 1 & 4 5 8 & 1 2 5 & 2 & 8 10 5 & 1 3 10 & 2 & 8 10 3 & 1 3 8 \\ \hline
(0,1,7,11,2,9,10,6,8) & 2 & 9 11 2 & 1 3 6 & 1 & 9 10 2 & 1 2 6 & 1 & 9 10 6 & 1 2 10 \\ \hline
\end{tabular}
\label{9--12--b}
\end{subtable}

\vspace{0.2 cm}

\begin{subtable}{1\textwidth}
\centering
\caption{The last three triplets of the base cycles, denoted by 7-8-9, 8-9-1 and 9-1-2.}
\begin{tabular}{|c||f|c|c|f|c|c|f|c|c|}
\hline
\multirow{2}{*}{base cycle} & \multicolumn{3}{|c}{(7-8-9)}& \multicolumn{3}{|c}{(8-9-1)} & \multicolumn{3}{|c|}{(9-1-2)} \\ \cline{2-10}
& d & triplet & type & d & triplet & type & d & triplet & type \\ \hline \hline
(0,1,2,4,5,8,10,3,6) & 3 & 3 6 10 & 1 4 8 & 3 & 0 3 6 & 1 4 7 & 1 & 0 1 6 & 1 2 7 \\ \hline
(0,1,7,11,2,9,10,6,8) & 2 & 6 8 10 & 1 3 5 & 2 & 6 8 0 & 1 3 7 & 1 & 0 1 8 & 1 2 9 \\ \hline
\end{tabular}
\label{9--12--c}
\end{subtable}
\end{table}


\begin{table}[htbp]
\caption{Base cycles and triplet type calculations of system $\CCC{9}{15}$ }

\begin{subtable}{1\textwidth}
\centering
\caption{The first three triplets of the base cycles, denoted by 1-2-3, 2-3-4 and 3-4-5.}
\begin{tabular}{|c||f|c|c|f|c|c|f|c|c|}
\hline
\multirow{2}{*}{base cycle} & \multicolumn{3}{|c}{(1-2-3)}& \multicolumn{3}{|c}{(2-3-4)} & \multicolumn{3}{|c|}{(3-4-5)} \\ \cline{2-10}
& d & triplet & type & d & triplet & type & d & triplet & type \\ \hline \hline
\multicolumn{10}{|l|}{Base cycle generating 5 cycles:} \\ \hline
(0,1,2,5,6,7,10,11,12) & 1 & 0 1 2 & 1 2 3 & 1 & 1 2 5 & 1 2 5 & 1 & 5 6 2 & 1 2 13 \\ \hline
\multicolumn{10}{|l|}{Base cycles generating 15 cycles:} \\ \hline
(0,1,3,4,8,2,5,7,9) & 1 & 0 1 3 & 1 2 4 & 1 & 3 4 1 & 1 2 14 & 1 & 3 4 8 & 1 2 6 \\ \hline
(0,1,6,3,10,2,12,4,7) & 1 & 0 1 6 & 1 2 7 & 2 & 1 3 6 & 1 3 6 & 3 & 3 6 10 & 1 4 8 \\ \hline
(0,1,10,14,6,12,2,5,11) & 1 & 0 1 10 & 1 2 11 & 2 & 14 1 10 & 1 3 12 & 4 & 6 10 14 & 1 5 9 \\ \hline
\end{tabular}
\label{9--15--a}
\end{subtable}

\vspace{0.5 cm}

\begin{subtable}{1\textwidth}
\centering
\caption{The next three triplets of the base cycles, denoted by 4-5-6, 5-6-7 and 6-7-8.}
\begin{tabular}{|c||f|c|c|f|c|c|f|c|c|}
\hline
\multirow{2}{*}{base cycle} & \multicolumn{3}{|c}{(4-5-6)}& \multicolumn{3}{|c}{(5-6-7)} & \multicolumn{3}{|c|}{(6-7-8)} \\ \cline{2-10}
& d & triplet & type & d & triplet & type & d & triplet & type \\ \hline \hline
(0,1,2,5,6,7,10,11,12) & 1 & 5 6 7 & 1 2 3 & 1 & 6 7 10 & 1 2 5 & 1 & 10 11 7 & 1 2 13 \\ \hline
 & \multicolumn{1}{l|}{} & \multicolumn{1}{l|}{} & \multicolumn{1}{l|}{} & \multicolumn{1}{l|}{} & \multicolumn{1}{l|}{} & \multicolumn{1}{l|}{} & \multicolumn{1}{l|}{} & \multicolumn{1}{l|}{} & \multicolumn{1}{l|}{} \\ \hline
(0,1,3,4,8,2,5,7,9) & 2 & 2 4 8 & 1 3 7 & 3 & 2 5 8 & 1 4 7 & 2 & 5 7 2 & 1 3 13 \\ \hline
(0,1,6,3,10,2,12,4,7) & 1 & 2 3 10 & 1 2 9 & 2 & 10 12 2 & 1 3 8 & 2 & 2 4 12 & 1 3 11 \\ \hline
(0,1,10,14,6,12,2,5,11) & 2 & 12 14 6 & 1 3 10 & 4 & 2 6 12 & 1 5 11 & 3 & 2 5 12 & 1 4 11 \\ \hline
\end{tabular}
\label{9--15--b}
\end{subtable}

\vspace{0.5 cm}

\begin{subtable}{1\textwidth}
\centering
\caption{The last three triplets of the base cycles, denoted by 7-8-9, 8-9-1 and 9-1-2.}
\begin{tabular}{|c||f|c|c|f|c|c|f|c|c|}
\hline
\multirow{2}{*}{base cycle} & \multicolumn{3}{|c}{(7-8-9)}& \multicolumn{3}{|c}{(8-9-1)} & \multicolumn{3}{|c|}{(9-1-2)} \\ \cline{2-10}
& d & triplet & type & d & triplet & type & d & triplet & type \\ \hline \hline
(0,1,2,5,6,7,10,11,12) & 1 & 10 11 12 & 1 2 3 & 1 & 11 12 0 & 1 2 5 & 1 & 0 1 12 & 1 2 13 \\ \hline
 & \multicolumn{1}{l|}{} & \multicolumn{1}{l|}{} & \multicolumn{1}{l|}{} & \multicolumn{1}{l|}{} & \multicolumn{1}{l|}{} & \multicolumn{1}{l|}{} & \multicolumn{1}{l|}{} & \multicolumn{1}{l|}{} & \multicolumn{1}{l|}{} \\ \hline
(0,1,3,4,8,2,5,7,9) & 2 & 5 7 9 & 1 3 5 & 2 & 7 9 0 & 1 3 9 & 1 & 0 1 9 & 1 2 10 \\ \hline
(0,1,6,3,10,2,12,4,7) & 3 & 4 7 12 & 1 4 9 & 3 & 4 7 0 & 1 4 12 & 1 & 0 1 7 & 1 2 8 \\ \hline
(0,1,10,14,6,12,2,5,11) & 3 & 2 5 11 & 1 4 10 & 4 & 11 0 5 & 1 5 10 & 1 & 0 1 11 & 1 2 12 \\ \hline
\end{tabular}
\label{9--15--c}
\end{subtable}
\end{table}


\begin{table}[htbp]
\caption{Base cycles and triplet type calculations of system $\CCC{9}{18}$ }

\begin{subtable}{1\textwidth}
\centering
\caption{The first three triplets of the base cycles, denoted by 1-2-3, 2-3-4 and 3-4-5.}
\begin{tabular}{|c||f|c|c|f|c|c|f|c|c|}
\hline
\multirow{2}{*}{base cycle} & \multicolumn{3}{|c}{(1-2-3)}& \multicolumn{3}{|c}{(2-3-4)} & \multicolumn{3}{|c|}{(3-4-5)} \\ \cline{2-10}
& d & triplet & type & d & triplet & type & d & triplet & type \\ \hline \hline
(0,1,2,4,5,8,3,6,7) & 1 & 0 1 2 & 1 2 3 & 1 & 1 2 4 & 1 2 4 & 1 & 4 5 2 & 1 2 17 \\ \hline
(0,1,5,3,9,4,2,10,8) & 1 & 0 1 5 & 1 2 6 & 2 & 1 3 5 & 1 3 5 & 2 & 3 5 9 & 1 3 7 \\ \hline
(0,1,9,4,8,10,3,6,13) & 1 & 0 1 9 & 1 2 10 & 3 & 1 4 9 & 1 4 9 & 1 & 8 9 4 & 1 2 15 \\ \hline
(0,1,10,3,7,14,17,2,11) & 1 & 0 1 10 & 1 2 11 & 2 & 1 3 10 & 1 3 10 & 3 & 7 10 3 & 1 4 15 \\ \hline
(0,3,13,17,5,10,14,4,9) & 3 & 0 3 13 & 1 4 14 & 4 & 13 17 3 & 1 5 9 & 4 & 13 17 5 & 1 5 11 \\ \hline
\end{tabular}
\label{9--18--a}
\end{subtable}

\vspace{0.5 cm}

\begin{subtable}{1\textwidth}
\centering
\caption{The next three triplets of the base cycles, denoted by 4-5-6, 5-6-7 and 6-7-8.}
\begin{tabular}{|c||f|c|c|f|c|c|f|c|c|}
\hline
\multirow{2}{*}{base cycle} & \multicolumn{3}{|c}{(4-5-6)}& \multicolumn{3}{|c}{(5-6-7)} & \multicolumn{3}{|c|}{(6-7-8)} \\ \cline{2-10}
& d & triplet & type & d & triplet & type & d & triplet & type \\ \hline \hline
(0,1,2,4,5,8,3,6,7) & 1 & 4 5 8 & 1 2 5 & 2 & 3 5 8 & 1 3 6 & 2 & 6 8 3 & 1 3 16 \\ \hline
(0,1,5,3,9,4,2,10,8) & 1 & 3 4 9 & 1 2 7 & 2 & 2 4 9 & 1 3 8 & 2 & 2 4 10 & 1 3 9 \\ \hline
(0,1,9,4,8,10,3,6,13) & 2 & 8 10 4 & 1 3 15 & 2 & 8 10 3 & 1 3 14 & 3 & 3 6 10 & 1 4 8 \\ \hline
(0,1,10,3,7,14,17,2,11) & 4 & 3 7 14 & 1 5 12 & 3 & 14 17 7 & 1 4 12 & 3 & 14 17 2 & 1 4 7 \\ \hline
(0,3,13,17,5,10,14,4,9) & 5 & 5 10 17 & 1 6 13 & 4 & 10 14 5 & 1 5 14 & 4 & 10 14 4 & 1 5 13 \\ \hline
\end{tabular}
\label{9--18--b}
\end{subtable}

\vspace{0.5 cm}

\begin{subtable}{1\textwidth}
\centering
\caption{The last three triplets of the base cycles, denoted by 7-8-9, 8-9-1 and 9-1-2.}
\begin{tabular}{|c||f|c|c|f|c|c|f|c|c|}
\hline
\multirow{2}{*}{base cycle} & \multicolumn{3}{|c}{(7-8-9)}& \multicolumn{3}{|c}{(8-9-1)} & \multicolumn{3}{|c|}{(9-1-2)} \\ \cline{2-10}
& d & triplet & type & d & triplet & type & d & triplet & type \\ \hline \hline
(0,1,2,4,5,8,3,6,7) & 1 & 6 7 3 & 1 2 16 & 1 & 6 7 0 & 1 2 13 & 1 & 0 1 7 & 1 2 8 \\ \hline
(0,1,5,3,9,4,2,10,8) & 2 & 8 10 2 & 1 3 13 & 2 & 8 10 0 & 1 3 11 & 1 & 0 1 8 & 1 2 9 \\ \hline
(0,1,9,4,8,10,3,6,13) & 3 & 3 6 13 & 1 4 11 & 5 & 13 0 6 & 1 6 12 & 1 & 0 1 13 & 1 2 14 \\ \hline
(0,1,10,3,7,14,17,2,11) & 3 & 17 2 11 & 1 4 13 & 2 & 0 2 11 & 1 3 12 & 1 & 0 1 11 & 1 2 12 \\ \hline
(0,3,13,17,5,10,14,4,9) & 5 & 4 9 14 & 1 6 11 & 4 & 0 4 9 & 1 5 10 & 3 & 0 3 9 & 1 4 10 \\ \hline
\end{tabular}
\label{9--18--c}
\end{subtable}
\end{table}


\begin{table}[htbp]
\caption{Base cycles and triplet type calculations of system $\CCC{9}{21}$ }

\begin{subtable}{1\textwidth}
\centering
\caption{The first three triplets of the base cycles, denoted by 1-2-3, 2-3-4 and 3-4-5.}
\begin{tabular}{|c||f|c|c|f|c|c|f|c|c|}
\hline
\multirow{2}{*}{base cycle} & \multicolumn{3}{|c}{(1-2-3)}& \multicolumn{3}{|c}{(2-3-4)} & \multicolumn{3}{|c|}{(3-4-5)} \\ \cline{2-10}
& d & triplet & type & d & triplet & type & d & triplet & type \\ \hline \hline
(0,1,2,4,5,8,3,6,7) & 1 & 0 1 2 & 1 2 3 & 1 & 1 2 4 & 1 2 4 & 1 & 4 5 2 & 1 2 20 \\ \hline
(0,1,5,3,9,4,2,10,8) & 1 & 0 1 5 & 1 2 6 & 2 & 1 3 5 & 1 3 5 & 2 & 3 5 9 & 1 3 7 \\ \hline
(0,1,9,4,8,10,3,6,12) & 1 & 0 1 9 & 1 2 10 & 3 & 1 4 9 & 1 4 9 & 1 & 8 9 4 & 1 2 18 \\ \hline
(0,1,10,3,7,11,2,6,13) & 1 & 0 1 10 & 1 2 11 & 2 & 1 3 10 & 1 3 10 & 3 & 7 10 3 & 1 4 18 \\ \hline
(0,1,11,3,8,13,2,4,14) & 1 & 0 1 11 & 1 2 12 & 2 & 1 3 11 & 1 3 11 & 3 & 8 11 3 & 1 4 17 \\ \hline
(0,3,6,18,4,13,1,5,14) & 3 & 0 3 6 & 1 4 7 & 3 & 3 6 18 & 1 4 16 & 2 & 4 6 18 & 1 3 15 \\ \hline
(0,3,13,7,12,20,6,17,11) & 3 & 0 3 13 & 1 4 14 & 4 & 3 7 13 & 1 5 11 & 1 & 12 13 7 & 1 2 17 \\ \hline
\end{tabular}
\label{9--21--a}
\end{subtable}

\vspace{0.5 cm}

\begin{subtable}{1\textwidth}
\centering
\caption{The next three triplets of the base cycles, denoted by 4-5-6, 5-6-7 and 6-7-8.}
\begin{tabular}{|c||f|c|c|f|c|c|f|c|c|}
\hline
\multirow{2}{*}{base cycle} & \multicolumn{3}{|c}{(4-5-6)}& \multicolumn{3}{|c}{(5-6-7)} & \multicolumn{3}{|c|}{(6-7-8)} \\ \cline{2-10}
& d & triplet & type & d & triplet & type & d & triplet & type \\ \hline \hline
(0,1,2,4,5,8,3,6,7) & 1 & 4 5 8 & 1 2 5 & 2 & 3 5 8 & 1 3 6 & 2 & 6 8 3 & 1 3 19 \\ \hline
(0,1,5,3,9,4,2,10,8) & 1 & 3 4 9 & 1 2 7 & 2 & 2 4 9 & 1 3 8 & 2 & 2 4 10 & 1 3 9 \\ \hline
(0,1,9,4,8,10,3,6,12) & 2 & 8 10 4 & 1 3 18 & 2 & 8 10 3 & 1 3 17 & 3 & 3 6 10 & 1 4 8 \\ \hline
(0,1,10,3,7,11,2,6,13) & 4 & 3 7 11 & 1 5 9 & 4 & 7 11 2 & 1 5 17 & 4 & 2 6 11 & 1 5 10 \\ \hline
(0,1,11,3,8,13,2,4,14) & 5 & 3 8 13 & 1 6 11 & 5 & 8 13 2 & 1 6 16 & 2 & 2 4 13 & 1 3 12 \\ \hline
(0,3,6,18,4,13,1,5,14) & 5 & 13 18 4 & 1 6 13 & 3 & 1 4 13 & 1 4 13 & 4 & 1 5 13 & 1 5 13 \\ \hline
(0,3,13,7,12,20,6,17,11) & 5 & 7 12 20 & 1 6 14 & 6 & 6 12 20 & 1 7 15 & 3 & 17 20 6 & 1 4 11 \\ \hline
\end{tabular}
\label{9--21--b}
\end{subtable}

\vspace{0.5 cm}

\begin{subtable}{1\textwidth}
\centering
\caption{The last three triplets of the base cycles, denoted by 7-8-9, 8-9-1 and 9-1-2.}
\begin{tabular}{|c||f|c|c|f|c|c|f|c|c|}
\hline
\multirow{2}{*}{base cycle} & \multicolumn{3}{|c}{(7-8-9)}& \multicolumn{3}{|c}{(8-9-1)} & \multicolumn{3}{|c|}{(9-1-2)} \\ \cline{2-10}
& d & triplet & type & d & triplet & type & d & triplet & type \\ \hline \hline
(0,1,2,4,5,8,3,6,7) & 1 & 6 7 3 & 1 2 19 & 1 & 6 7 0 & 1 2 16 & 1 & 0 1 7 & 1 2 8 \\ \hline
(0,1,5,3,9,4,2,10,8) & 2 & 8 10 2 & 1 3 16 & 2 & 8 10 0 & 1 3 14 & 1 & 0 1 8 & 1 2 9 \\ \hline
(0,1,9,4,8,10,3,6,12) & 3 & 3 6 12 & 1 4 10 & 6 & 0 6 12 & 1 7 13 & 1 & 0 1 12 & 1 2 13 \\ \hline
(0,1,10,3,7,11,2,6,13) & 4 & 2 6 13 & 1 5 12 & 6 & 0 6 13 & 1 7 14 & 1 & 0 1 13 & 1 2 14 \\ \hline
(0,1,11,3,8,13,2,4,14) & 2 & 2 4 14 & 1 3 13 & 4 & 0 4 14 & 1 5 15 & 1 & 0 1 14 & 1 2 15 \\ \hline
(0,3,6,18,4,13,1,5,14) & 4 & 1 5 14 & 1 5 14 & 5 & 0 5 14 & 1 6 15 & 3 & 0 3 14 & 1 4 15 \\ \hline
(0,3,13,7,12,20,6,17,11) & 5 & 6 11 17 & 1 6 12 & 4 & 17 0 11 & 1 5 16 & 3 & 0 3 11 & 1 4 12 \\ \hline
\end{tabular}
\label{9--21--c}
\end{subtable}
\end{table}


\begin{table}[htbp]
\caption{Base cycles and triplet type calculations of system $\CCC{9}{24}$ }

\begin{subtable}{1\textwidth}
\centering
\caption{The first three triplets of the base cycles, denoted by 1-2-3, 2-3-4 and 3-4-5.}
\begin{tabular}{|c||f|c|c|f|c|c|f|c|c|}
\hline
\multirow{2}{*}{base cycle} & \multicolumn{3}{|c}{(1-2-3)}& \multicolumn{3}{|c}{(2-3-4)} & \multicolumn{3}{|c|}{(3-4-5)} \\ \cline{2-10}
& d & triplet & type & d & triplet & type & d & triplet & type \\ \hline \hline
\multicolumn{10}{|l|}{Base cycle generating 8 cycles:} \\ \hline
(0,1,2,8,9,10,16,17,18) & 1 & 0 1 2 & 1 2 3 & 1 & 1 2 8 & 1 2 8 & 1 & 8 9 2 & 1 2 19 \\ \hline
\multicolumn{10}{|l|}{Base cycles generating 24 cycles:} \\ \hline
(0,1,3,4,7,2,5,6,10) & 1 & 0 1 3 & 1 2 4 & 1 & 3 4 1 & 1 2 23 & 1 & 3 4 7 & 1 2 5 \\ \hline
(0,1,6,5,13,2,3,11,12) & 1 & 0 1 6 & 1 2 7 & 1 & 5 6 1 & 1 2 21 & 1 & 5 6 13 & 1 2 9 \\ \hline
(0,1,14,3,5,7,11,2,15) & 1 & 0 1 14 & 1 2 15 & 2 & 1 3 14 & 1 3 14 & 2 & 3 5 14 & 1 3 12 \\ \hline
(0,1,17,3,5,10,2,6,19) & 1 & 0 1 17 & 1 2 18 & 2 & 1 3 17 & 1 3 17 & 2 & 3 5 17 & 1 3 15 \\ \hline
(0,2,8,5,1,7,9,16,12) & 2 & 0 2 8 & 1 3 9 & 3 & 2 5 8 & 1 4 7 & 3 & 5 8 1 & 1 4 21 \\ \hline
(0,2,10,7,1,12,5,14,19) & 2 & 0 2 10 & 1 3 11 & 3 & 7 10 2 & 1 4 20 & 3 & 7 10 1 & 1 4 19 \\ \hline
(0,3,9,15,1,4,14,20,10) & 3 & 0 3 9 & 1 4 10 & 6 & 3 9 15 & 1 7 13 & 6 & 9 15 1 & 1 7 17 \\ \hline
(0,3,12,15,4,22,13,5,17) & 3 & 0 3 12 & 1 4 13 & 3 & 12 15 3 & 1 4 16 & 3 & 12 15 4 & 1 4 17 \\ \hline
(0,4,9,16,1,11,18,5,13) & 4 & 0 4 9 & 1 5 10 & 5 & 4 9 16 & 1 6 13 & 7 & 9 16 1 & 1 8 17 \\ \hline
\end{tabular}
\label{9--24--a}
\end{subtable}

\vspace{0.5 cm}

\begin{subtable}{1\textwidth}
\centering
\caption{The next three triplets of the base cycles, denoted by 4-5-6, 5-6-7 and 6-7-8.}
\begin{tabular}{|c||f|c|c|f|c|c|f|c|c|}
\hline
\multirow{2}{*}{base cycle} & \multicolumn{3}{|c}{(4-5-6)}& \multicolumn{3}{|c}{(5-6-7)} & \multicolumn{3}{|c|}{(6-7-8)} \\ \cline{2-10}
& d & triplet & type & d & triplet & type & d & triplet & type \\ \hline \hline
\multicolumn{10}{|l|}{Base cycle generating 8 cycles:} \\ \hline
(0,1,2,8,9,10,16,17,18) & 1 & 8 9 10 & 1 2 3 & 1 & 9 10 16 & 1 2 8 & 1 & 16 17 10 & 1 2 19 \\ \hline
\multicolumn{10}{|l|}{Base cycles generating 24 cycles:} \\ \hline
(0,1,3,4,7,2,5,6,10) & 2 & 2 4 7 & 1 3 6 & 2 & 5 7 2 & 1 3 22 & 1 & 5 6 2 & 1 2 22 \\ \hline
(0,1,6,5,13,2,3,11,12) & 3 & 2 5 13 & 1 4 12 & 1 & 2 3 13 & 1 2 12 & 1 & 2 3 11 & 1 2 10 \\ \hline
(0,1,14,3,5,7,11,2,15) & 2 & 3 5 7 & 1 3 5 & 2 & 5 7 11 & 1 3 7 & 4 & 7 11 2 & 1 5 20 \\ \hline
(0,1,17,3,5,10,2,6,19) & 2 & 3 5 10 & 1 3 8 & 3 & 2 5 10 & 1 4 9 & 4 & 2 6 10 & 1 5 9 \\ \hline
(0,2,8,5,1,7,9,16,12) & 2 & 5 7 1 & 1 3 21 & 2 & 7 9 1 & 1 3 19 & 2 & 7 9 16 & 1 3 10 \\ \hline
(0,2,10,7,1,12,5,14,19) & 5 & 7 12 1 & 1 6 19 & 4 & 1 5 12 & 1 5 12 & 2 & 12 14 5 & 1 3 18 \\ \hline
(0,3,9,15,1,4,14,20,10) & 3 & 1 4 15 & 1 4 15 & 3 & 1 4 14 & 1 4 14 & 6 & 14 20 4 & 1 7 15 \\ \hline
(0,3,12,15,4,22,13,5,17) & 6 & 22 4 15 & 1 7 18 & 6 & 22 4 13 & 1 7 16 & 7 & 22 5 13 & 1 8 16 \\ \hline
(0,4,9,16,1,11,18,5,13) & 5 & 11 16 1 & 1 6 15 & 7 & 11 18 1 & 1 8 15 & 6 & 5 11 18 & 1 7 14 \\ \hline
\end{tabular}
\label{9--24--b}
\end{subtable}
\end{table}

\begin{table}
\ContinuedFloat
\begin{subtable}{1\textwidth}
\centering
\caption{The last three triplets of the base cycles, denoted by 7-8-9, 8-9-1 and 9-1-2.}
\begin{tabular}{|c||f|c|c|f|c|c|f|c|c|}
\hline
\multirow{2}{*}{base cycle} & \multicolumn{3}{|c}{(7-8-9)}& \multicolumn{3}{|c}{(8-9-1)} & \multicolumn{3}{|c|}{(9-1-2)} \\ \cline{2-10}
& d & triplet & type & d & triplet & type & d & triplet & type \\ \hline \hline
\multicolumn{10}{|l|}{Base cycle generating 8 cycles:} \\ \hline
(0,1,2,8,9,10,16,17,18) & 1 & 16 17 18 & 1 2 3 & 1 & 17 18 0 & 1 2 8 & 1 & 0 1 18 & 1 2 19 \\ \hline
\multicolumn{10}{|l|}{Base cycles generating 24 cycles:} \\ \hline
(0,1,3,4,7,2,5,6,10) & 1 & 5 6 10 & 1 2 6 & 4 & 6 10 0 & 1 5 19 & 1 & 0 1 10 & 1 2 11 \\ \hline
(0,1,6,5,13,2,3,11,12) & 1 & 11 12 3 & 1 2 17 & 1 & 11 12 0 & 1 2 14 & 1 & 0 1 12 & 1 2 13 \\ \hline
(0,1,14,3,5,7,11,2,15) & 4 & 11 15 2 & 1 5 16 & 2 & 0 2 15 & 1 3 16 & 1 & 0 1 15 & 1 2 16 \\ \hline
(0,1,17,3,5,10,2,6,19) & 4 & 2 6 19 & 1 5 18 & 5 & 19 0 6 & 1 6 12 & 1 & 0 1 19 & 1 2 20 \\ \hline
(0,2,8,5,1,7,9,16,12) & 3 & 9 12 16 & 1 4 8 & 4 & 12 16 0 & 1 5 13 & 2 & 0 2 12 & 1 3 13 \\ \hline
(0,2,10,7,1,12,5,14,19) & 5 & 14 19 5 & 1 6 16 & 5 & 14 19 0 & 1 6 11 & 2 & 0 2 19 & 1 3 20 \\ \hline
(0,3,9,15,1,4,14,20,10) & 4 & 10 14 20 & 1 5 11 & 4 & 20 0 10 & 1 5 15 & 3 & 0 3 10 & 1 4 11 \\ \hline
(0,3,12,15,4,22,13,5,17) & 4 & 13 17 5 & 1 5 17 & 5 & 0 5 17 & 1 6 18 & 3 & 0 3 17 & 1 4 18 \\ \hline
(0,4,9,16,1,11,18,5,13) & 5 & 13 18 5 & 1 6 17 & 5 & 0 5 13 & 1 6 14 & 4 & 0 4 13 & 1 5 14 \\ \hline
\end{tabular}
\label{9--24--c}
\end{subtable}
\end{table}


\begin{table}[htbp]
\caption{Base cycles and triplet type calculations of system $\CCC{9}{27}$ }

\begin{subtable}{1\textwidth}
\centering
\caption{The first three triplets of the base cycles, denoted by 1-2-3, 2-3-4 and 3-4-5.}
\begin{tabular}{|c||f|c|c|f|c|c|f|c|c|}
\hline
\multirow{2}{*}{base cycle} & \multicolumn{3}{|c}{(1-2-3)}& \multicolumn{3}{|c}{(2-3-4)} & \multicolumn{3}{|c|}{(3-4-5)} \\ \cline{2-10}
& d & triplet & type & d & triplet & type & d & triplet & type \\ \hline \hline
(0,1,3,6,10,17,21,24,26) & 1 & 0 1 3 & 1 2 4 & 2 & 1 3 6 & 1 3 6 & 3 & 3 6 10 & 1 4 8 \\ \hline
(0,2,6,1,7,20,26,21,25) & 2 & 0 2 6 & 1 3 7 & 1 & 1 2 6 & 1 2 6 & 1 & 6 7 1 & 1 2 23 \\ \hline
(0,3,4,10,1,26,17,23,24) & 1 & 3 4 0 & 1 2 25 & 1 & 3 4 10 & 1 2 8 & 3 & 1 4 10 & 1 4 10 \\ \hline
(0,4,9,1,8,19,26,18,23) & 4 & 0 4 9 & 1 5 10 & 3 & 1 4 9 & 1 4 9 & 1 & 8 9 1 & 1 2 21 \\ \hline
(0,5,7,13,2,25,14,20,22) & 2 & 5 7 0 & 1 3 23 & 2 & 5 7 13 & 1 3 9 & 5 & 2 7 13 & 1 6 12 \\ \hline
(0,6,10,18,1,26,9,17,21) & 4 & 6 10 0 & 1 5 22 & 4 & 6 10 18 & 1 5 13 & 8 & 10 18 1 & 1 9 19 \\ \hline
(0,7,9,19,5,22,8,18,20) & 2 & 7 9 0 & 1 3 21 & 2 & 7 9 19 & 1 3 13 & 4 & 5 9 19 & 1 5 15 \\ \hline
(0,8,9,18,4,23,9,18,19) & 1 & 8 9 0 & 1 2 20 & 1 & 8 9 18 & 1 2 11 & 5 & 4 9 18 & 1 6 15 \\ \hline
(0,10,3,15,1,26,12,24,17) & 3 & 0 3 10 & 1 4 11 & 5 & 10 15 3 & 1 6 21 & 2 & 1 3 15 & 1 3 15 \\ \hline
(0,11,1,22,8,19,5,26,16) & 1 & 0 1 11 & 1 2 12 & 6 & 22 1 11 & 1 7 17 & 6 & 22 1 8 & 1 7 14 \\ \hline
(0,12,9,1,16,11,26,18,15) & 3 & 9 12 0 & 1 4 19 & 3 & 9 12 1 & 1 4 20 & 7 & 9 16 1 & 1 8 20 \\ \hline
(0,13,4,19,20,7,8,23,14) & 4 & 0 4 13 & 1 5 14 & 6 & 13 19 4 & 1 7 19 & 1 & 19 20 4 & 1 2 13 \\ \hline
\end{tabular}%
\label{9--27--b}
\end{subtable}
\end{table}

\begin{table}
\ContinuedFloat
\begin{subtable}{1\textwidth}
\centering
\caption{The last three triplets of the base cycles, denoted by 7-8-9, 8-9-1 and 9-1-2.}
%
\label{9--30--a}
\end{subtable}
\end{table}

\begin{table}
\ContinuedFloat
\begin{subtable}{1\textwidth}
\centering
\caption{The next three triplets of the base cycles, denoted by 4-5-6, 5-6-7 and 6-7-8.}
%
\label{9--30--b}
\end{subtable}
\end{table}

\begin{table}
\ContinuedFloat
\begin{subtable}{1\textwidth}
\centering
\caption{The last three triplets of the base cycles, denoted by 7-8-9, 8-9-1 and 9-1-2.}
%
\label{9--30--c}
\end{subtable}
\end{table}


\end{document}